\documentclass[12pt]{article}
\date{}
\usepackage{marvosym, fontawesome, parskip}
\usepackage[title]{appendix}
\usepackage{xcolor}
\usepackage[margin=1in]{geometry}                
\usepackage{graphicx}
\usepackage{subcaption}
\usepackage{pdflscape}
\usepackage{amssymb}
\usepackage[normalem]{ulem}
\usepackage{hyperref}
\usepackage{enumitem}
\usepackage{mathrsfs}
\usepackage{epstopdf}
\usepackage{rotating}
\usepackage{longtable} 
\usepackage{adjustbox}
\usepackage{float}

\usepackage{color}
\usepackage{bbm, dsfont}
\usepackage{pst-node}
\usepackage{tikz-cd}
\usepackage{amsfonts} 
\usepackage{geometry}
\usepackage{amsthm}
\usepackage{amsmath}
\usepackage{titlesec}
\usepackage{amssymb}
\usepackage{enumitem}
\usepackage{float}
\usepackage [english]{babel}
\usepackage [autostyle, english = american]{csquotes}
\usepackage{algorithm}
\usepackage[noend]{algpseudocode} 
\makeatletter
\def\BState{\State\hskip-\ALG@thistlm}
\makeatother
\usepackage{hyperref}
\usepackage[normalem]{ulem}
\usepackage{mathrsfs}
\usepackage[italicdiff]{physics}
\usepackage[doi=false, url =false , isbn = false,style=alphabetic,sorting=nyt, backend = biber, maxcitenames=50, maxalphanames = 5, maxnames=50]{biblatex}
\newlist{casess}{enumerate}{1}
\setlist[casess]{label=     \textbf{Case} \arabic*:}
\usepackage{mathtools}

\makeatletter
\newcommand*{\rom}[1]{\expandafter\@slowromancap\romannumeral #1@}
\makeatother

\usepackage{etoolbox}

\makeatletter
\patchcmd{\ttlh@hang}{\parindent\z@}{\parindent\z@\leavevmode}{}{}
\patchcmd{\ttlh@hang}{\noindent}{}{}{}
\makeatother

\usepackage{listings}
\usepackage{color} 
\definecolor{mygreen}{RGB}{28,172,0} 
\definecolor{mylilas}{RGB}{170,55,241}

\newlist{Assumptions}{enumerate}{1}
\setlist[Assumptions]{label=     \textbf{Assumption} \arabic*:}

\makeatletter

\newsavebox{\@brx}
\newcommand{\llangle}[1][]{\savebox{\@brx}{\(\m@th{#1\langle}\)}%
  \mathopen{\copy\@brx\kern-0.5\wd\@brx\usebox{\@brx}}}
\newcommand{\rrangle}[1][]{\savebox{\@brx}{\(\m@th{#1\rangle}\)}%
  \mathclose{\copy\@brx\kern-0.5\wd\@brx\usebox{\@brx}}}
\makeatother

\usepackage{lipsum} 
\usepackage{titlesec}
\titleformat{\subsection}[runin]
       {\normalfont\bfseries}
       {\thesubsection}
       {0.5em}
       {}
       [.]

 \newtheorem{thm}{Theorem}[section]
 \newtheorem{cor}[thm]{Corollary}
 
 \newtheorem{lem}[thm]{Lemma}
 \newtheorem{prop}[thm]{Proposition}

\newtheorem{letterthm}{Theorem}

 \theoremstyle{definition}
 \newtheorem{defn}[thm]{Definition}
 \theoremstyle{remark}
 
 \newtheorem{ex}[thm]{Example}
 \numberwithin{equation}{section}

\numberwithin{equation}{section}

\newcommand{\actson}{\curvearrowright}

\DeclareMathOperator{\ab}{ab}

\newcommand{\cG}{\mathcal{G}}
\newcommand{\cU}{\mathcal{U}}
\newcommand{\cW}{\mathcal{W}}

\newcommand{\cS}{\mathcal{S}}
\newcommand{\cB}{\mathcal{B}}

\newcommand{\St}{\rm St}

\newcommand{\HNN}{\rm HNN}
\newcommand{\GL}{\rm GL}
\newcommand{\PSL}{\rm PSL}
\newcommand{\Cost}{\rm Cost}
\newcommand{\BS}{\rm BS}
\newcommand{\rmK}{\rm K}

\def\N{\mathbb{N}}
\def\T{\mathbb{T}}
\def\Z{\mathbb{Z}}

\def\Z{\mathbb Z}
\newcommand{\Ad}{\rm Ad}
\newcommand{\Rad}{\rm Rad}
\newcommand{\rmH}{\rm H}
\def\S{\mathcal S}

\newcommand{\Aut}{{\rm Aut}}
\newcommand{\Ext}{\rm Ext}
\newcommand{\Hom}{\rm Hom}
\newcommand{\SL}{\rm SL}

\newcommand{\id}{\rm id}
\newcommand{\rmZ}{\rm Z}

\DeclareMathOperator{\Span}{Span}
\DeclarePairedDelimiterX{\inp}[2]{\langle}{\rangle}{#1, #2}

\makeatletter
\newcommand*\bigcdot{\mathpalette\bigcdot@{.5}}
\newcommand*\bigcdot@[2]{\mathbin{\vcenter{\hbox{\scalebox{#2}{$\m@th#1\bullet$}}}}}
\makeatother

\def\r{{\rm r}}

\def\<{\langle}
\def\>{\rangle}

\newcommand{\cR}{\mathcal{R}}

\numberwithin{equation}{section}

\begin{document}

\title{Connes-embeddability and twisted group von Neumann algebras}

\author{
  Soham Chakraborty\textsuperscript{1} \and
  Felipe Flores \textsuperscript{2}
}

\date{\today}

\setlength{\parindent}{0em}

\maketitle 

\begingroup
\renewcommand\thefootnote{\arabic{footnote}}
\footnotetext[1]{\hspace{0 em} \faMapMarker : Départment de Mathématiques et Applications, École Normale Supérieure, 45 Rue d'Ulm, 75005 Paris, France. \Letter: \texttt{soham.chakraborty@ens.psl.eu}}

\footnotetext[2]{\hspace{0 em} \faMapMarker : Department of Mathematics, University of Virginia, 114 Kerchof Hall. 141 Cabell Dr, Charlottesville, Virginia, United States. \Letter: \texttt{hmy3tf@virginia.edu}}
\endgroup

\begin{abstract}\setlength{\parindent}{0pt}\setlength{\parskip}{1ex}\noindent
    We call a countable discrete group \textit{cocycle-hyperlinear} if all its twisted group von Neumann algebras are Connes-embeddable. We show that the class of cocycle-hyperlinear groups includes extensions of treeable groups by amenable groups and central extensions of $\SL(n,\Z), \GL(n,\Z)$ and $\PSL(n,\Z)$ for $n \geq 2$. We also show that this class is closed under taking direct products, free products, amalgamated free products with amenable amalgams, HNN extensions by amenable groups, as well as certain quotients, extensions, and overgroups.
\end{abstract}


\section{Introduction}

Finite-dimensional approximation has linked operator algebras and group theory since the seminal works of Murray and von Neumann. Let $R$ denote the hyperfinite II$_1$ factor and let $R^\omega$ be a tracial ultrapower. The breakthrough work of Connes \cite{Connes_classification} isolated the famous \textit{embedding problem}, which asks if every II$_1$ factor with a separable predual has a trace-preserving embedding into $R^\omega$. For a countable discrete group $G$, the corresponding problem asks if the group von Neumann algebra $L(G)$ embeds tracially into $R^\omega$; such groups are called \textit{hyperlinear}. The terminology was introduced by Radulescu in \cite{Radulescu}, who proved that the non-residually finite Baumslag-Solitar group $\langle a,b \; | \; ab^3a^{-1} = b^2 \rangle $ is hyperlinear.

A parallel approximation theory arose from symbolic dynamics. Gromov introduced \textit{sofic} groups in \cite{Gromov_99} as a common generalization of amenable and residually finite groups. The theory was systematically developed by Weiss in \cite{Weiss00}; Elek and Szabo then showed in \cite{Elek_Szabo} that sofic groups are hyperlinear. Thus, we have the implications: 
\begin{align*}
    \text{residually finite or amenable} \implies \text{ sofic } \implies \text{ hyperlinear}
\end{align*}
The reverse implication hyperlinear $\implies$ sofic is still open. Existence of non-sofic or non-hyperlinear groups remained a challenging open question for several decades. Note that the recent negative resolution of the Connes embedding problem in \cite{10.1145/3485628} (see also \cite{Manzoor}) does not exhibit group von Neumann algebras that do not embed into $R^\omega$. However, in the course of the past few months, several breakthrough results were obtained in this domain using LLMs. Building on work by Kun in \cite{Kun19} and Kun-Thom in \cite{Thom_Kun}, OpenAI announced the construction of an explicit non-sofic group \cite[Chapter 3]{OpenAI}, thereby resolving the longstanding existence problem for non-sofic groups. Even more recently, non-hyperlinear groups were shown to exist by Alekseev, Liu and Thom in \cite{Thom_etal}. They solved the centralizer problem (\cite[Problem 6.2(a)]{Alekseev_Thom}) implying the existence of a non-hyperlinear group, which was already shown by Thom in \cite{Thom_centralizer}). We refer the reader to Pestov's survey \cite{Pestov_survey} for a historical overview and an elaborate treatment of these notions. We also refer the reader to some interesting results extending the class of hyperlinear groups in \cite{Hayes_Sale,Gao_Sri_Patchell,Gao_Sri_Patchell26}.

The purpose of this article is to study a strengthening of hyperlinearity that deals with central extensions and projective representations. Given a normalized scalar 2-cocycle $\omega \in \rmZ^2(G,\T)$, let $L_{\omega}(G)$ denote the cocycle-twisted group von Neumann algebra generated by unitaries $\{u_g \; | \; g \in G\}$ satisfying: 
\begin{align*}
    u_g u_h = \omega(g,h)u_{gh} \text{ for all } g,h \in G \text{ and } \Tr(u_g) = \delta_{g,e}
\end{align*}
We call a countable discrete group $G$ \textit{cocycle-hyperlinear} if $L_{\omega}(G)$ is Connes-embeddable for all normalized 2-cocycles. The definition is motivated by Thom's criterion for central extensions in \cite{Thom10}. Indeed, if $1 \rightarrow A 
\rightarrow \widetilde{G} \rightarrow G \rightarrow 1$ is a central extension of $G$ by a countable abelian group $A$ represented by a 2-cocycle $\sigma \in \rmZ^2(G,A)$, then Thom observed in \cite[Lemma 3.4]{Thom10} that $\widetilde{G}$ is hyperlinear if and only if $L_{\chi \circ \sigma}(G)$ is Connes-embeddable for all characters $\chi \in \widehat{A}$. Consequently, $G$ is cocycle-hyperlinear if and only if every central extension of $G$ by a countable abelian group is hyperlinear. Thus, twisted group von Neumann algebras don't just pose an auxiliary variant of the problem, but they precisely provide a way to test the hyperlinearity of central extensions.

This point of view also provides context for this article in the historical search for non-hyperlinear groups. In \cite{Kirschberg94}, Kirchberg showed that his factorization property is equivalent to residual finiteness for property (T) groups. Deligne's non-residually finite extensions of arithmetic symplectic groups \cite{Deligne} thus became natural testing grounds for the lack of finite-dimensional approximations. Failure of residual finiteness or Kirchberg's factorization property does not imply failure of hyperlinearity (see \cite{Thom10}). Nevertheless, these Deligne-type extensions remained one of the principal sources of candidates (see \cite{DeChiffre_Glebsky_Thom_Lubotsky}). In \cite{Becker_Lubotsky}, Becker and Lubotzky showed that an infinite property (T) hyperlinear group cannot be Hilbert-Schmidt stable, and introduced the flexible variant that later became central to candidate constructions. In \cite{Dogon23}, Dogon showed that flexible Hilbert Schmidt stability of $\text{Sp}_{2n}(\Z)$ would force a non-hyperlinear central extension, and then Dogon and Vigdorovich obtained analogous conditional finite central extensions for higher-rank $S$-arithmetic lattices, including $\SL(2, \Z[1/p])$ in \cite{Dogon_Vigdorovich}.

These works explain why the seemingly elementary-looking question of whether finite central extensions of hyperlinear groups are hyperlinear (explicitly posed by Arzhantseva-Belai-Finn-Sell-Glebsky in \cite[Question 4.1]{Arzhantseva_etal}) captures a genuine frontier in our understanding of Connes-embeddability. Our results provide positive answers to this question in several different situations. Recall that a group $G$ is called \textit{treeable} if it admits an essentially free p.m.p. action $G \actson (X,\mu)$ such that the orbit equivalence relation is treeable. Our first main result, that follows from Proposition \ref{Prop: eq conditions} and Theorem \ref{Thm: action with treeable eq relation and amenable stabilizer} is: 
\begin{letterthm}
\label{Main Thm: amenable-by-treeable}
 Let $G$ be a countable discrete group, then the following are equivalent: 
\begin{enumerate}
    \item There is a p.m.p. action $G \actson (X,\mu)$ whose associated equivalence relation $\mathcal R$ is treeable and the stabilizers $G_x$ are amenable for a.e. $x \in X$,
    \item $G$ is an extension of a treeable group by an amenable group,
    \item The quotient $G/\Rad(G)$ is treeable where $\Rad(G)$ is the amenable radical of $G$.
\end{enumerate} 
 Moreover if $G$ satisfies the above equivalent conditions, then $G$ is cocycle-hyperlinear.  
\end{letterthm}

The proof of Theorem \ref{Main Thm: amenable-by-treeable} involves the theory of discrete measured groupoids. Indeed, we show that the twisted von Neumann algebra of the transformation groupoid can be written as the von Neumann algebra of a twisted semi-direct product of the equivalence relation acting on the isotropy bundle. The von Neumann algebra of the isotropy bundle is hyperfinite since the stabilizers are amenable, and the treeability of the equivalence relation provides a homomorphic section into the groupoid by the main result of \cite{Kida_treeability}, which allows us to untwist the cocycle (see also \cite{Popa_Vaes_Shlyakhtenko}). Once again using the treeability of the equivalence relation, we reduce Connes-embeddability of the groupoid von Neumann algebra to permanence under amalgamated free products over a hyperfinite von Neumann algebra (see \cite{Brown_Dykema_Jung}). This simultaneously treats all cocycle twists and is hence stronger than proving hyperlinearity of abelian central extensions. The next main result that we prove in Theorem \ref{Thm: SL(n,Z) is cocycle-hyperlinear} and Corollary \ref{Cor: GLn and PSLn} is: 

\begin{letterthm}
\label{Main Thm: central extensions of SL(n,Z)}
    For $n \geq 2$, the groups $\SL(n,\Z)$, $\GL(n,\Z)$ and $\PSL(n,\Z)$ are cocycle-hyperlinear. Moreover, all central extensions of these groups are cocycle-hyperlinear. 
\end{letterthm}

Since $\SL(2,\Z)$ is treeable, Theorem \ref{Main Thm: amenable-by-treeable} applies for $n = 2$. However, for $n \geq 3$, the main step is to observe that reduction modulo 4 gives an isomorphism of the Schur multipliers, i.e., 
\begin{align*}
    \rmH_2(\SL(n,\Z), \Z) \cong \rmH_2(\SL(n, \Z/4\Z), \Z)
\end{align*}
For $n \geq 5$, this follows from the Steinberg extension and the standard computation of Milnor's $\rmK_2$ groups for $\Z$ and $\Z/4\Z$ (see \cite{MilnorBook} and \cite{Weibel}). For $n = 3,4$, we use van der Kallen's computation \cite{VanderKallen} of the Schur multipliers and an analogous `reduction modulo 4' observation. Perfectness and the universal coefficient theorem then imply that every scalar 2-cocycle on $\SL(n,\Z)$ is cohomologous to one coming from the finite quotient $\SL(n, \Z/4\Z)$. Finally we observe that for a hyperlinear group, any twisted group von Neumann algebra associated to a cocycle coming from a finite quotient is Connes-embeddable. The statement about $\GL(n,\Z)$, $\PSL(n,\Z)$ and central extensions follow from the more general next result recording several permanence properties of this class.

\begin{letterthm}
    \label{Main thm: permanence properties}
    The class of cocycle-hyperlinear groups satisfy the following permanence properties: 
    \begin{enumerate}
        \item If $(G_n)_{n \in N}$ is an increasing sequence of cocycle-hyperlinear groups, then the directed union $\bigcup_n G_n$ is cocycle-hyperlinear (Lemma \ref{Lemma: directed union})
        \item If $(G_n)_{n \in \N}$ is a sequence of cocycle-hyperlinear groups, the countable direct sum $\bigoplus G_n$ is cocycle-hyperlinear (Corollary \ref{Cor: countable direct sum})
        \item If $N < G$ is either a finite normal subgroup or a central normal subgroup and $G$ is cocycle-hyperlinear, then $G/N$ is cocycle-hyperlinear (Propositions \ref{Prop: quotients by finite subgroups} and \ref{Prop: quotient by central subgroup})
        \item If $Q$ is a perfect cocycle-hyperlinear group and $1 \rightarrow K \rightarrow G \rightarrow Q \rightarrow 1$ is a central extension, then $G$ is cocycle-hyperlinear (Proposition \ref{Prop: extensions of perfect groups})
        \item If $H \leq G$ is a co-amenable subgroup and $H$ is cocycle-hyperlinear, then $G$ is cocycle-hyperlinear. In particular if $G$ has a finite index cocycle-hyperlinear subgroup, it is cocycle-hyperlinear. (Proposition \ref{Prop: coamenable overgroups})
        \item If $(G_n)_{n \in \N}$ is a sequence of cocycle-hyperlinear groups and $L$ is a common amenable subgroup, the amalgamated free product $\ast_L G_n$ is cocycle-hyperlinear. In particular, free products of cocycle-hyperlinear groups are cocycle-hyperlinear. (Proposition \ref{Prop: amalagamted free products over amenable groups}) 
        \item Let $G$ be cocycle-hyperlinear and $H \leq G$ be amenable. Letting $\nu: H \rightarrow G$ be an injective homomorphism, the HNN-extension $\HNN(G,H,\nu)$ is cocycle-hyperlinear. In particular, the Baumslag-Solitar groups $\BS(m,n)$ are cocycle-hyperlinear. (Proposition \ref{Prop: HNN extensions SH} and Corollary \ref{Cor: Baumslag Solitar groups SH})  
    \end{enumerate}
\end{letterthm}

Thom's criterion repeatedly converts the above statements involving extensions into uniform control over all character twists. For extensions of perfect groups, we use the universal Schur covering. For amalgamated free products and HNN-extensions, we use existing results on Connes-embeddability of amalagamated free products over amenable amalgams in \cite{Brown_Dykema_Jung} and more generally, on Connes-embeddability of graphs of von Neumann algebras over amenable edges, as in \cite{Fima_Amaury}. Hence this class of groups is robust under constructions that naturally generate new groups. We remark that our methods however do not deduce non-hyperlinearity of the candidates that appear in \cite{Arzhantseva_etal}, \cite{Dogon23}, \cite{Dogon_Vigdorovich}; in particular stable hyperlinearity of $\SL(2, \Z[1/p])$ and $\text{Sp}_{2n}(\Z)$ remains open.

\textbf{2020 Mathematics Subject Classification:} Primary 46L10, Secondary 22D25, 37A15, 19C09.
\newline
\textbf{Key Words:} Hyperlinear groups, Connes embedding problem, treeable group, soecial linear group, central extension, vanishing cohomology.

\section{Preliminaries}

Let $R$ be the hyperfinite II$_1$ factor and let $\omega$ be a non-principal ultrafilter on $\N$. We call a tracial von Neumann algebra \textit{Connes-embeddable} if it has a trace-preserving embedding into the ultrapower $R^\omega$. Recall from \cite{Radulescu} that a countable discrete group $G$ is called \textit{hyperlinear} if $L(G)$ is Connes-embeddable. The main definition of this article is:  

\begin{defn}
    We will call a countable discrete group \textit{cocycle-hyperlinear} if for any normalized Borel 2-cocycle $\omega \in \rmZ^2(G,\T)$, the twisted group von Neumann algebra $L_\omega(G)$ is Connes-embeddable. 
\end{defn}

The next few subsections very briefly introduce some preliminaries that we use for our arguments. 

\subsection{Discrete measured groupoids and their von Neumann algebras}

We recall that a groupoid $\cG$ over the unit space $X\equiv\cG^{(0)}$ is a small category in which all the morphisms (arrows) are invertible. 
The domain and range maps are denoted by $s,t:\cG\to X$ and the family of composable pairs by $\,\cG^{(2)}\!\subset\cG\times\cG$.  For $g\in\cG\,,\,A, B\subset\cG$ we use the notations
\begin{equation*}\label{botations}
AB\coloneqq\{ab\,:\,a\in A,b\in B, s(a)=t(b)\},
\end{equation*}
\begin{equation*}
Ag\coloneqq A\{g\}\quad \textup{and}\quad gA\coloneqq\{g\}A
\end{equation*}
Note that these sets could be empty in non-trivial situations. Moreover, for $x \in X$, one uses the standard notations 
\begin{equation*}\label{faneaka}
\cG_x\!\coloneqq s^{-1}(x)\,,\quad\cG^x\!\coloneqq t^{-1}(x)\,,\quad\cG_x^x\coloneqq\cG_A\cap\cG^A.
\end{equation*}

From now on, we will suppose that $\cG$ is a groupoid equipped with the structure of a standard Borel space such that the source, target, composition, and inverse maps are Borel, and the source and target maps are countable-to-one. Furthermore, let $\mu$ be a probability measure on the set of units $X$. Then, for any measurable subset $A\subset\cG$, the functions  
$X\ni x\mapsto \#\bigl (s^{-1}(x)\cap A\bigr )$ and $X\ni x\mapsto \#\bigl (t^{-1}(x)\cap A\bigr )$ are measurable. We say that $\mu$ is {\it invariant}, or that $(\cG,\mu)$ is a {\it discrete probability measure-preserving groupoid} if the equality  
$$
\int_{X} \#\bigl (\cG^x\cap A\bigr )d\mu (x)=\int_{X} \#\bigl (\cG_x\cap A\bigr )d\mu (x)
$$
holds for all Borel subsets $A\subset \cG$. If that is the case, then we define the Borel measure $\mu_\cG(A):=\int_{X} \#\bigl (\cG^x\cap A\bigr )d\mu (x)$, which is clearly $\sigma$-finite. 

A \textit{bisection} is a Borel subset where $s$ and $t$ are injective. The bisections modulo null sets and with multiplication as defined above form an inverse semigroup, called the \textit{full pseudogroup} of $\cG$ and denoted by $[[\cG]]$. Each such bisection $V \in [[\cG]]$ defines a measure preserving bijection between $s(V)$ and $t(V)$. 

A countable measured p.m.p. equivalence relation $\cR$ on a standard probability space $(X,\mu)$ is viewed as a discrete measured groupoid which is principal, i.e., $\cR_x^x$ is trivial for all $x \in X$. Given a discrete measured groupoid, there is an associated equivalence relation given by $\{(t(g),s(g) \; | \; g \in \cG\}$ and an associated isotropy bundle $\cG_0$ which is a measured field of isotropy groups $(\cG^x_x)_{x \in X}$. For a p.m.p. action $G \actson (X,\mu)$ of a countable discrete group, the transformation groupoid is a discrete measured groupoid denoted by $X \rtimes G$, where the isotropy groups correspond to the stabilizers of the action.

A \textit{graphing} of $\cR$ is a countable family $\Phi = \{\phi_n\}$ of bisections in $[[\cR]]$ such that the graphs generate $\cR$ as an equivalence relation. The \textit{cost} of the graphing $\Phi$ is given by $\Cost(\Phi) = \sum_n \mu(s(\phi_n))$. The cost of the equivalence relation $\cR$ is given by $\Cost(\cR) = \inf \{\Cost(\Phi) \; | \; \Phi \text{ is a graphing of } \cR\}$. The graphing $\Phi$ is called a \textit{treeing} if every non-empty reduced word $\phi$ formed from the letters in the set $\Phi \cup \Phi^{-1}$ has a null-set of fixed points, i.e., $\mu(\{x \in X \; | \; \phi(x) = x\}) = 0$. The equivalence relation $\cR$ is called \textit{treeable} if $\cR$ admits a treeing. Gaboriau showed in \cite[Theorem IV.1]{Gaboriau_cost} that if $\Phi$ is a treeing of $\cR$, then $\Cost(\cR) = \Cost(\Phi)$. 

An equivalence relation is called \textit{hyperfinite} if it can be written as a countable union of equivalence relations with finite orbits almost everywehere. It turns out that every hyperfinite equivalence relation is treeable by \cite[Proposition III.3(2)]{Gaboriau_cost}. We refer the reader to \cite{Kosaki} for an exposition on free products of measured equivalence relations. Vaguely, the notation $\cR = \cR_1 \ast \cR_2$ means that $\cR_1$ and $\cR_2$ generate $\cR$ and no non-trivial alternating product of non-unit arrows is a unit. We have the following: 

\begin{defn}
    A countable discrete group $G$ is called \textit{treeable} if it admits an essentially free p.m.p. action whose orbit equivalence relation is treeable. $G$ is called \textit{strongly treeable} if every essentially free p.m.p. action of $G$ has a treeable orbit equivalence relation. $G$ is called \textit{anti-treeable} if none of its essentially free p.m.p. actions are treeable. 
\end{defn}

A \textit{normalized scalar 2-cocycl}e on $\cG$ is a Borel map $\omega: \cG \rightarrow \T$ satisftying on composable pairs: 
\begin{align*}
    \omega(g,h)\omega(gh,k) = \omega(g,hk)\omega(hk) \text{ and } \omega(t(g),g) = \omega(g,s(g)) = 1
\end{align*}
We denote the cocycles by $\rmZ^2(\cG,\T)$. A cocycle $\omega$ is a \textit{coboundary} if there is a Borel map $b: \cG \rightarrow \T$ such that $\omega(g,h) = b(g)b(h) \overline{b(gh)}$ for composable pairs a.e. and $b|_X = 1$. The corresponding cohomology group is denoted by $\rmH^2(\cG,\T)$. The following is immediate from \cite{Kida_treeability}. 

\begin{prop}
\label{Prop: treeable groups are SH}
    Let $G$ be a treeable group. Then $G$ is cocycle-hyperlinear.
\end{prop}
\begin{proof}
    Let $\omega \in \rmZ^2(G,\T)$ be a normalized cocycle and let $\cR$ be a treeable orbit equivalence relation of a $G$-action. Define $\widetilde{\omega} \in \rmZ^2(\cR,\T)$ by: 
    \begin{align*}
        \widetilde{\omega}((hgx,gx)(gx,x)) = \omega(h,g)
    \end{align*}
    Notice that up to measure zero, this is well defined by essential freeness of the action. By \cite[Corollary 2.8]{Kida_treeability}, $\rmH^2(\cR,\T) = \{0\}$ and hence $L(\cR) \cong L_\omega(\cR)$. By the main result of \cite{Elek_Lippner}, $L(\cR)$ is Connes-embeddable and the result follows. 
\end{proof}

An action $\alpha$ of a discrete measured groupoid $\cG$ with unit space $X$ on a measured field of tracial (finite) von Neumann algebras $(A_x.\tau_x)_{x \in X}$ with separable preudals is a measured field of *-isomorphisms $\alpha_g: A_{s(g)} \rightarrow A_{t(g)}$ with $\alpha_g \circ \alpha_h = \alpha_{gh}$ for a.e. composable pair and $\alpha_x = \id$ for a.e. $x \in X$. Given a 2-cocycle $\omega: \cG \rightarrow \T$ and letting $A$ be the direct integral $\int^\oplus_X A_x d\mu(x)$, one defines the crossed product $A \rtimes_{\alpha,\rho} \cG$ as in \cite[Definition 4.5]{Chak25}.

\subsection{Homology: perfect groups and Schur multipliers}

We use the standard notation for group homology and cohomology. Recall that $\rmH_1(G,\Z)$ is naturally identified with the \textit{abelianization} of $G$ and $\rmH_2(G,\Z)$ is called the \textit{Schur multipliers} on $G$ and $\rmH_1(G,\Z)$. If $\rmH_1(G,\Z)$ is trivial, the group is called \textit{perfect}. It is well known that for $n \geq 3$, $\SL(n,\Z)$ is perfect. Notice that the universal coefficient theorem gives: 
\begin{align*}
    1 \rightarrow \Ext_{\Z}^1(\rmH_{1}(G,\T)) \rightarrow \rmH^2(G, \T) \rightarrow \Hom(\rmH_2(G,\Z),\T) \rightarrow 1
\end{align*}
Since $\T$ is a divisible abelian group, it is injective as a $\Z$-module, and hence $\Ext^1_{\Z}(\rmH_1(G,\T))$ vanishes, and we obtain the natural isomorphism $\rmH^2(G,\T) \cong \Hom(\rmH_2(G,\Z),\T)$.

We now recall the definition of a Schur covering. Suppose now that $\widetilde{G}$ is a central extension of a countable group $G$ of the form:
\begin{align*}
    1 \rightarrow \rmH_2(G,\Z) \rightarrow \widetilde{G} \xrightarrow{p} G \rightarrow 1,
\end{align*}
where $a \in \rmH^2(G,\rmH_2(G,\Z))$ is the cohomology class associated with this extension. The universal coefficient theorem gives a natural surjection $\phi: \rmH^2(G,\rmH_2(G,\Z)) \rightarrow \Hom(\rmH_2(G,\Z),\rmH_2(G,\Z))$. If $\phi(a)$ is the identity map on $\rmH_2(G,\Z)$, then the extension $\widetilde{G}$ is called a \textit{Schur covering group} of $G$. If $G$ is a perfect group, then there is a unique Schur covering group $\widetilde{G}$ up to isomorphism. In that case, the extension is called the \textit{universal central extension}. It is universal in the following sense: if there is another central extension $1 \rightarrow K \rightarrow E \xrightarrow{q} G \rightarrow 1$, then there is a unique homomorphism $\phi: \widetilde{G} \rightarrow E$ that makes the following diagram commute: 

\[ \begin{tikzcd}
\widetilde{G} \arrow{r}{p} \arrow[swap]{d}{\phi} & G \arrow{d}{\id} \\%
E \arrow{r}{q}& G
\end{tikzcd} \]

It turns out that the extension $\widetilde{G}$ is automatically perfect and even has trivial Schur multiplier group, i.e., $\rmH_2(\widetilde{G},\Z) = \{0\}$.

\subsection{Milnor's K$_2$ computations for special linear groups}
This subsection contains all the necessary homological and K-theoretic preliminaries particularly necessary for Section \ref{Sec: Special Linear groups}. We only introduce the notions we need; hence, this subsection is rather ad-hoc. For a more elaborate exposition, we direct the reader to \cite{MilnorBook}.

For a unital ring $A$, let $\GL(n,A)$ denote the general linear group. Consider the elementary matrices $E_{ij}$ and  $e_{ij}(a) = I + aE_{ij}$. Let $E_n(A) < \GL(n,A)$ be the subgroup generated by $\{e_{ij}(a) \; | \; a \in A\}$. If $n \geq 3$, one checks that for distinct $i,j,k$ we have: 
\begin{align*}
    [e_{ik}(a), e_{kj}(b)] = e_{ij}(ab)
\end{align*}
In particular, $e_{ij}(a)$ is a commutator for all $i,j$ and $a \in A$. Thus, $[E_{n}(A),E_n(A)] = E_n(A)$ implies that $E_n(A)$ is perfect. In this article, we shall only deal with the rings $\Z$ and $\Z/4\Z$. By integer row reduction, one obtains $E_n(\Z) = \SL(n,\Z)$, and by Gaussian elimination over a local ring, one obtains $E_n(\Z/4\Z) = \SL(n, \Z/4\Z)$. Moreover, this shows that the reduction map: 
\begin{align*}
    r_4: \SL(n,\Z) \rightarrow \SL(n,\Z/4\Z)
\end{align*}
is surjective. Indeed, every element of $\SL(n,\Z/4\Z)$ is a product of elementary matrices, and for each $e_{ij}(\overline{a})$ with $\overline{a} \in \Z/4\Z$, one can pick a lift $a \in \Z$ such that $r_4(e_{ij}(a)) = e_{ij}(\overline{a})$. In fact, $\SL(n,\Z)$ is residually finite as the intersection of the kernels of the maps $r_{4^k}: \SL(n,\Z) \rightarrow \SL(n, \Z/ 4^k \Z)$ over all $k$ is trivial.  

The following is directly taken from \cite[Definitions III.5.1 and III.5.2]{Weibel}. For each unital ring $A$, recall that the \textit{Steinberg group} denoted by $\St_n(A)$ is generated by symbols $\{x_{ij}(a) \; | \; a \in A\}$ with the following relations: 
\begin{align*}
    x_{ij}(a)x_{ij}(b) &= x_{ij}(a + b) \text{ for all } i,j \text{ and } a,b \in A \\
    [x_{ij}(a), x_{kl}(b)] &= \begin{cases}
        1 & j \neq k \text{ and } i \neq l \\
        x_{il}(ab) & j = k \text{ and } i \neq l \\
        x_{kj}(-ba) & j \neq k \text{ and } i = l
    \end{cases}
\end{align*}
There is a natural surjection $\phi_n: \St_n(A) \rightarrow E_n(A)$ sending $x_{ij}(a)$ to $e_{ij}(a)$ and the kernel is denoted by $\rmK_2(n,A)$. Since the relations for $n+1$ include the relations for $n$, there is a natural map $\St_n(A) \rightarrow \St_{n+1}(A)$. We denote by $\St(A)$ the direct limits of $\St_n(A)$ as $n \rightarrow \infty$. Similarly, let $E(A)$ denote the direct limits of $E_n(A)$; then there is an induced surjective homomorphism $\phi: \St(A) \rightarrow E(A)$. We denote by $\rmK_2(A)$ the kernel of $\phi$. Recall that for a commutative ring $A$, the \textit{Krull dimension} is given by: 
\begin{align*}
    \dim A = \sup \{r \; | \; P_0 \subset P_1 \subset ... \subset P_r \text{ are strict inclusions of prime ideals}\}
\end{align*}
One can easily check that $\dim \Z = 1$ and $\dim \Z/4\Z = 0$. We need the following standard results for our purposes: 
\begin{prop}
\label{Prop: random facts from Weibel's book}
    Let $A$ be a unital commutative ring. Then: 
    \begin{enumerate}
        \item If $n \geq 5$, then $\St_n(A) \rightarrow E_n(A)$ is the universal central extension of $E_n(A)$. Hence $\rmH_2(E_n(A), \Z) \cong \rmK_2(n,A)$
        \item If $A$ is noetherian with Krull dimension $d$, then $\rmK_2(n,A) \rightarrow \rmK_2(A)$ is an isomorphism for $n \geq d+3$
        \item The groups $\rmK_2(\Z)$ and $\rmK_2(\Z/4\Z)$ are isomorphic to $\Z/2 \Z$, generated by the element:
        \begin{align*}
            z_A = (x_{12}(1)x_{21}(-1)x_{12}(1))^4
        \end{align*}
    \end{enumerate}
\end{prop}
For the proofs of the above standard facts, we refer the reader to \cite{Weibel}. In particular, see \cite[Theorem III.5.4 and Proposition III.5.5.1]{Weibel} for 1, \cite[Remark III.5.5.2]{Weibel} for 2 and \cite[Example III.5.2.2, Theorem III.5.11.1 and Excercise III.5.13]{Weibel} for 3. 

\subsection{Actions with amenable stabilizers}

In this subsection, we prove the equivalence of the three conditions in Theorem \ref{Main Thm: amenable-by-treeable}. For a countable measured equivalence relation $\cR$ on a standard probability space $(X,\mu)$, we say that a Borel p.m.p. automorphism $\theta \in \Aut(X,\mu)$ \textit{normalizes} $\cR$ if $\theta$ satisfies: $(x,y) \in \cR \iff (\theta(x),\theta(y)) \in \cR$ for a.e. $x,y \in X$.

We need the following ad-hoc lemma to prove the equivalence.  

\begin{lem}
\label{Lemma: ad-hoc cost lemma}
    Let $\cR$ be a treeable p.m.p. equivalence relation on $(X,\mu)$ and let $\cS \subseteq \cR$ be an aperiodic hyperfinite equivalence relation. If $\cR$ is generated by $\cS$ and countably many p.m.p. automorphisms $\theta_n \in \Aut(X,\mu)$ that normalize $\cS$, then $\cR$ is hyperfinite. 
\end{lem}
\begin{proof}
    Since $\cS$ is hyperfinite, choose a graphing $\Phi$ of $\cS$ with cost 1 and let $\epsilon > 0$. Notice that $\cS \cap \theta_n \cS \theta_n^{-1} = \cS$ and is aperiodic. By applying \cite[Lemma V.3]{Gaboriau_cost}, there exists Borel subsets $A_n \subset X$ with $\mu(A_n) < \epsilon/ 2^n$ for each $n$ such that $\{\theta_n|_{A_n}\}$ and $\cS$ still generate $\cR$. As a consequence: 
    \begin{align*}
        1 \leq \Cost(\cR) \leq \Cost(\Phi) + \sum_n \mu(A_n) < 1 + \epsilon
    \end{align*}
    Thus $\Cost(\cR) = 1$ and hence $\cR$ is hyperfinite by \cite[Corollary IV.2]{Gaboriau_cost}. 
\end{proof}

\begin{prop}
\label{Prop: eq conditions}
    Let $G$ be a countable discrete group, then the following are equivalent:
    \begin{enumerate}
        \item There is a p.m.p. action $G \actson (X,\mu)$ whose stabilizers $G_x$ are amenable for a.e. $x \in X$
    \item $G$ is an extension of a treeable group by an amenable group
    \item The quotient $G/\Rad(G)$ is treeable where $\Rad(G)$ is the amenable radical of $G$.
    \end{enumerate}
\end{prop}
\begin{proof}
    For $2 \implies 1$, suppose that we have an exact sequence: 
    \begin{align*}
        1 \rightarrow N \rightarrow G \rightarrow Q \rightarrow 1
    \end{align*}
    where $Q$ is treeable and $N$ is amenable. Let $Q \actson (X,\mu)$ be an essentially free p.m.p. action that is treeable. Let $G \actson (X,\mu)$ by the quotient map. Since the $Q$-action is essentially free, the stabilizers $G_x$ for the $G$-action are isomorphic to $N$ for a.e. $x \in X$. Notice that $3 \implies 2$ is immediate as $G$ is an extension of $G/\Rad(G)$ by $\Rad(G)$. We now prove that $1 \implies 3$. Let $N = \Rad(G)$ and let $\cR$ be the orbit equivalence relation of $G \actson (X,\mu)$. If $G$ is amenable then there is nothing to prove, so we assume that $G$ is non-amenable. By \cite[Corollary 1.5]{BaderDuchesneLecreux}, for a.e. $x \in X$, the stabilizer subgroup $G_x$ is in $N$. Let $\cS$ be the orbit equivalence relation of $N \actson (X,\mu)$, which is hyperfinite because $N$ is amenable. Notice that the group elements $g \in G$ normalize $\cS$ as $N$ is a normal subgroup. 
    
    Let $X_\infty \subset X$ be the Borel subset $\{x \in X \; | \; |N \cdot x| = \infty\}$. Suppose that $\mu(X_\infty) > 0$, then $\cS|_{X_\infty}$ is aperiodic and hyperfinite and $\cR|_{X_\infty}$ is still treeable. The automorphisms given by the group elements $g \in G$ normalize $\cS|_{X_\infty}$ and together generate $\cR|_{X_\infty}$. By Lemma \ref{Lemma: ad-hoc cost lemma}, the orbit equivalence relation $\cR|_{X_\infty}$ is hyperfinite. Since the orbit equivalence relation is hyperfinite and the stabilizers are amenable a.e., in fact $G$ is amenable by \cite[Lemma 3.6]{Hjorth}. This gives a contradiction and hence $\mu(X_\infty) = 0$ and hence the $N$-orbits are finite a.e.  

    So the action $N \actson (X,\mu)$ has finite orbits, and hence a fundamental domain $Y \subset X$. Let $q: X \rightarrow Y$ be the Borel selector map and let $\nu = q_* \mu$ be the pushforward measure. Since $N$ is normal, we get an action of $G/N$ on $(Y,\nu)$ given by $gN \cdot q(x) = q(gx)$. Since the stabilizers of $G \actson X$ lie in $N$, the action of $G/N$ is free. But notice that the orbit equivalence relation of $G/N \actson (Y,\nu)$ is precisely $\cR|_{Y}$, which is treeable for the measure class of $\mu|_Y$. One checks now that since $\cS$ is p.m.p., for each Borel subset $E \subset X$ we have: 
    \begin{align*}
        \nu(E) = \mu(q^{-1}(E)) = \int_{E} |N \cdot y| \; d\mu(y)
    \end{align*}
    Since $|N \cdot x|$ is finite for a.e. $x \in X$, this shows that $\nu$ is equivalent to $\mu|_Y$, completing the proof.
\end{proof}

In the next section we deal with cocycle-hyperlinearity of groups satisfying these equivalent conditions, hence proving Theorem \ref{Main Thm: amenable-by-treeable}.

\section{Amenable-by-treeable groups}

Notice that if $\omega \in \rmZ^2(G,\T)$ and $\alpha: G \actson (X,\mu)$ is a p.m.p. action then this gives a cocycle action $(\alpha,\omega): G \actson (X,\mu)$ such that the twisted crossed product $L^\infty(X,\mu) \rtimes_{(\alpha,\omega)} G$ is defined by the relations: 
\begin{align*}
    u_g f u_g^* = g \cdot f \text{ and } u_gu_h = \omega(g,h) u_{gh}
\end{align*}
for all $g,h \in G$ and $f \in L^\infty(X,\mu)$. One checks that $L_\omega(G)$ tracially embeds into $L^\infty(X,\mu) \rtimes_{(\alpha,\omega)} G$. There is no essential freeness assumption required here. For a group action $G \actson (X,\mu)$, we shall denote the stabilizer at $x$ by $G_x$ and the isotropy bundle by $G_0 = (G_x)_{x \in X}$.

\begin{lem}
\label{Lemma: twisted transformation groupoid vna}
    Let $G \actson (X,\mu)$ be a p.m.p. action and $X \rtimes G$ be the transformation groupoid. Let $\omega \in \rmZ^2(X \rtimes G,\T)$ be a 2-cocycle and denote the restriction to $G_x$ by $\omega_x$. Let $\cR$ be the associated orbit equivalence relation and suppose that $\cR$ is treeable. Then there is an action $\delta: \cR \actson G_0$ such that as measured groupoids $X \rtimes G \cong G_0 \rtimes_\delta \cR$. Moreover, $\delta$ induces an untwisted action $\alpha$ of $\cR$ on the von Neumann algebra as below, and we have:  
    \begin{align*}
        L_{\omega}(X \rtimes G) \cong \left( \int^{\oplus}_{X} L_{\omega_x}(G_x) d\mu(x)\right) \rtimes_{\alpha} \cR
    \end{align*}
\end{lem}
\begin{proof}
    Since $\cR$ is treeable, by \cite[Proposition 6.5]{Popa_Vaes_Shlyakhtenko}, there is a homomorphic section $j: \cR \rightarrow X \rtimes G$ and an action $\delta: \cR \actson G_0$ given by $\delta_{(y,x)}: G_x \rightarrow G_y$ and $\delta_{(y,x)}(g) = j(y,x) g j(y,x)^{-1}$ for $g \in G_x$, such that the measured groupoids $X \rtimes G$ and $G_0 \rtimes \cR$ are isomorphic. Let $A_x = L_{\omega_x}(G_x)$ and $A = \int^{\oplus}_{X} A_xd\mu(x)$. For each $x \in X$ and $k \in G_x$, let us denote by $u^x_k$ the corresponding group unitary in $A_x$. Denote the image of an element $(g,(x,y)) \in G_0 \rtimes_\delta \cR$ in the twisted convolution algebra of the groupoid by $[g,(x,y)]$. For all $(y,x) \in \cR$ and $k \in G_x$, we define $d(k, (y,x)) = \omega \left( j(y,x), (k,(x,x)) \right) \overline{\omega\left((\delta_{(y,x)}(k),(y,y)),j(y,x)) \right)}$. Then we calculate: 
    \begin{align*}
        [j(y,x)]\cdot[(k,(x,x))] &= \omega(j(y,x),(k,(x,x)))[j(y,x)\cdot (k,(x,x))] \\ &= \omega(j(y,x),(k,(x,x)))[(\delta_{(y,x)}(k),(y,y))\cdot j(y,x)] \\ &= d(k,(y,x))\omega\left((\delta_{(y,x)}(k),(y,y)),j(y,x)) \right)[(\delta_{(y,x)}(k),(y,y))\cdot j(y,x)] \\&= d(k,(y,x))[(\delta_{(y,x)}(k),(y,y))]\cdot[j(y,x)]
    \end{align*}
    Now consider $x \in X$ and $k,l \in G_x$. Using the previous equation multiple times, we calculate: 
    \begin{align*}
        [j(y,x)]\left([(k,(x,x))][(l,(x,x))]\right) &= [j(y,x)]\omega_x(k,l)[(kl,(x,x))] \\ &= \omega_x(k,l)d(kl,(y,x))[(\delta_{(y,x)}(kl), (y,y))][j(y,x)] \text{ and} 
        \\
        \left([j(y,x)]
         [(k,(x,x))]\right) [(l,(x,x))] &= d(k,(y,x))[(\delta_{(y,x)}(k),(y,y))][j(y,x)][(l,(x,x))] \\ = d(k,(y,x)) &d(l,(y,x))[(\delta_{(y,x)}(k),(y,y))][(\delta_{(y,x)}(l),(y,y))][j(y,x)] \\ = d(k,(y,x))&d(l,(y,x)) \omega_y(\delta_{(y,x)}(k) ,\delta_{(y,x)}(l))[(\delta_{(y,x)}(kl),(y,y))][j(y,x)] 
    \end{align*}
    By associativity, we get for $x, y \in X$ and $k,l \in G_x$ that: 
    \begin{align}
    \label{Eq for d and omega}
        \omega_x(k,l)d(kl,(y,x)) = d(k,(y,x))d(l,(y,x))\omega_y(\delta_{(y,x)}(k),\delta_{(y,x)}(l))
    \end{align}
    Now let us define $\alpha_{(y,x)}: A_x \rightarrow A_y$ by $\alpha_{(y,x)}(u^x_k) = d(k,(y,x)) u^y_{\delta_{(y,x)}(k)}$. Notice that: 
    \begin{align*}
        \alpha_{(y,x)}(u^x_k)\alpha_{(y,x)} (u^x_l) &= d(k,(y,x))d(l,(y,x))\omega_y(\delta_{(y,x)}(k),\delta_{(y,x)}(l)) u^y_{\delta_{(y,x)}(kl)} \\&= \omega_x(k,l)d(kl,(y,x)) u^y_{\delta_{(y,x)}(kl)} \text{ by Eq. \ref{Eq for d and omega}} 
        \\ &= \omega_x(k,l)\alpha_{(y,x)}(u^x_{kl}) = \alpha_{(y,x)}(u^x_k u^x_l)
    \end{align*}
    As usual, this map between generators is trace preserving and extends to a normal *-isomorphism $\alpha_{(y,x)}: A_x \rightarrow A_y$. It is a tedious routine check that $(\alpha_{(y,x)})_{(y,x) \in \cR}$ is a measurable field of isomorphisms and we omit the proof. One checks from the definition that $d(\delta_{(y,x)}(k),(z,y)) d(k, (y,x)) = d(k, (z,x))$ for $x,y,z \in X$ and $k \in G_x$. Now notice that: 
    \begin{align*}
        \alpha_{(z,y)} \circ \alpha_{(y,x)}(u^x_k) &= d(k,(y,x)) \alpha_{(z,y)}(u^y_{\delta_{(y,x)}(k)}) = d(k,(y,x))d(\delta_{(y,x)}(k), (z,y))u^{y}_{\delta_{(z,y)} \circ  \delta_{(y,x)}(k)}  \\ &= d(k,(z,x))u^{y}_{\delta_{(z,x)}(k)} = \alpha_{(z,x)}(u^x_k)
    \end{align*}
    Now let $\rho \in \rmZ^2(\cR,\T)$ be a 2-cocycle defined by $\rho((z,y),(y,x)) = \omega(j(z,y),j(y,x))$. It is easy to see the cocycle identity since $j$ is a homomorphic section. For each bisection $V \in [[\cR]]$, consider the bisection $j(V) \in [[G_0 \rtimes \cR]]$. Denote the corresponding partial isometry in $L_{\omega}(G_0 \rtimes \cR)$ by $\cU_V$. For each $V \in [[\cR]]$, the action $\alpha$ induces a normal *-isomorphism $\alpha_V : A 1_{s(V)} \rightarrow A 1_{t(V)}$ which is defined fiberwise by $(\alpha_V(a))_{y} = \alpha_{(y,x)}(a_x)$ for all $(y,x) \in V$. Similarly for $V,W \in [[\cR]]$, we define the unitary $\rho_{V,W}$ in $\cU(A 1_{t(VW)})$ by $\rho_{V,W}(t(r,s)) = \rho(r,s)$. We claim that in $L_{\omega}(G_0 \rtimes \cR)$, for all $V \in [[\cR]]$ we have: 
    \begin{align}
    \label{Eq: crossed product 1}
        \cU_V^* \cU_V = 1_{s(V)} &\text{  and  } \cU_V \cU_V^* = 1_{t(V)} \\
 \label{Eq: crossed product 2}        \cU_V a \cU_V^* = \alpha_V(a) & \text{  for all  }  a \in A 1_{s(V)}  \\ \label{Eq: crossed product 3}
        \cU_V \cU_W= \rho_{V,W} &\cU_{VW} \\
 \label{Eq: crossed product 4}        E_A(a \cU_V) = a 1_{V \cap X} & \text{  for all  } a \in A 1_{s(V)}
    \end{align}
    Notice that Equations \ref{Eq: crossed product 1}, \ref{Eq: crossed product 2} and \ref{Eq: crossed product 3} follow immediately from the pointwise calculations before. Moreover $j(V) \cap G_0 = V \cap X$ and hence $E(\cU_V) = 1_{V \cap X}$, thus verifying Equation \ref{Eq: crossed product 4} and the claim. As in \cite[Definitions 4.5 and 4.6]{Chak25} and the discussion below \cite[Definition 4.6]{Chak25}, these four equations are the defining properties of the cocycle crossed product $A \rtimes_{(\alpha,\rho)} \cR$. Now we claim that $L_\omega(G_0 \rtimes \cR)$ is generated by $A$ and the partial isometries $(\cU_V)_{V \in [[\cR]]}$. Suppose the claim is true, since the crossed product is the unique von Neumann algebra where the defining properties are satisfied, there is a trace preserving isomorphism $L_{\omega}(G_0 \rtimes \cR) \cong A \rtimes_{(\alpha,\rho)} \cR$ such that the expectations onto $A$ are intertwined by the isomorphism.
    
    To prove the claim, notice that it is enough to show that every bisection $B \in [[G_0 \rtimes \cR]]$ is generated by bisections in $[[G_0]]$ and $(j(V))_{V \in [[\cR]]}$. Let $q: [[G_0 \rtimes \cR]] \rightarrow [[\cR]]$ be the canonical quotient map given by $q(B) = \{(t(g),s(g)) \;  | \; g \in B\}$. Let $B_0 \in [[G_0 \rtimes \cR]]$ and consider the bisection $q(B_0) \in [[\cR]]$. Let $V_0 = j(q(B_0))$ and let $\kappa$ be the map on $B_0$ given by $b \mapsto b j(q(b))^{-1}$. One checks that that $\kappa$ is a Borel map on $B_0$ and $\kappa(B_0) \in [[G_0]]$. By construction we get $B_0 = \kappa(B_0) V_0$, and hence verifying the claim and thus $L_\omega(G_0 \rtimes \cR) \cong A \rtimes_{(\alpha,\rho)} \cR$.

    Finally, since $\cR$ is treeable, $\rmH^2(\cR,\T) = \{0\}$ by \cite[Corollary 2.8]{Kida_treeability} and hence there is a measurable function $b: \cR \rightarrow \T$ such that $\rho(r,s) = b(r)b(s)\overline{b(rs)}$ for a.e. $r,s \in \cR$. Consider the Borel function $a: G_0 \rtimes \cR \rightarrow \T$ given by $a(g) = \overline{b(q(g))}$. Now let us define the cocycle $\omega'$ on the groupoid $G_0 \rtimes \cR$ by $\omega'(g,h) = a(g)a(h)\overline{a(gh)} \omega(g,h)$. By definition, $\omega$ and $\omega'$ are cohomologous, and hence $L_{\omega}(G_0 \rtimes \cR) \cong L_{\omega'}(G_0 \rtimes \cR)$. Let $\rho' = \omega' \circ j$ and notice that for $r,s \in \cR$: 
    \begin{align*}
        \rho'(r,s) = \omega'(j(r).j(s)) = \overline{b(r)b(s)}b(rs) \omega(j(r),j(s)) = \overline{b(r)b(s)}b(rs)\rho(r,s) = 1
    \end{align*}
    Notice $\omega'$ and $\omega$ coincide on the isotropy subgroupoid $G_0$. Moreover, one checks that the cocycle $d$ and the action $\alpha$ are equal for $\omega$ and $\omega'$. Therefore, from the first part of the proof, we get: 
    \begin{align*}
       A \rtimes_{\alpha} \cR  \cong L_{\omega'}(G_0 \rtimes \cR) \cong L_{\omega}(G_0 \rtimes \cR) \cong A \rtimes_{(\alpha,\rho)} \cR
    \end{align*}
    thus completing the proof. 
\end{proof}

\begin{lem}
\label{Lemma: free products and actions}
    Let $\cR = \cR_1 \ast \cR_2$ be equivalence relations on $(X,\mu)$ and let $(A_x)_{x \in X}$ be a measured field of finite hyperfinite von Neumann algebras. Let $A$ be the direct integral, and let $\alpha: \cR \actson A$ be an action; then: 
    \begin{align*}
        A \rtimes_{\alpha} \cR \cong (A \rtimes_\alpha \cR_1)  \ast_A (A \rtimes_\alpha \cR_2)
    \end{align*}
\end{lem}
\begin{proof}
    Let $M = A \rtimes_{\alpha} \cR$, and $M_i = A \rtimes_{\alpha} \cR_i$ for $i \in \{1,2\}$. Let $E: M \rightarrow A$ be the unique trace preserving conditional expectation, and let $E_1$ and $E_2$ denote the restrictions to $M_1$ and $M_2$. It is enough to prove that $M_1$ and $M_2$ generate $M$ and are freely independent over $A$. Let $\cB_1 = \{\phi_0 = X, \phi_1, \phi_2....\}$ and $\cB_2 = \{\psi_0 = X, \psi_1,\psi_2,...\}$ be countable bases for the equivalence relations $\cR_1$ and $\cR_2$, respectively. Let $\cW$ be the set of finite alternating words from $\cB_1$ and $\cB_2$. Each element $\rho \in \cW$ is a Borel bisection; hence, $\rho \in [[\cR]]$. Moreover, after possibly discarding a measure zero set, every element $(x,y) \in \cR$ can be written as $(x,x_1)(x_1,x_2),...,(x_n,y)$ for elements alternating in $\cR_1$ and $\cR_2$. Since $\cW$ is countable, and $\cB_i$ are bases for $\cR_i$, once again after possibly removing a null set, we get that $\cR = \bigcup_{\rho \in \cW} \rho$. Therefore, $A$ and the partial isometries $\{u_\rho \;|\; \rho \in \cW\}$ generate $M$. Since each $u_\rho$ is a product of partial isometries occurring in $M_1$ and $M_2$, we conclude that $u_\rho \in W^*(M_1,M_2)$. This immediately gives that $W^*(M_1,M_2) = M$.   

    Now we proceed to prove free independence. For $i \in \{1,2\}$, let $D_i < M_i$ be the *-algebra generated by $A$ and the partial isometries corresponding to $\cB_i$. Since $E_{i}(\phi) = \phi \cap X$ for a bisection $\phi \in [[\cR_i]]$, we have that: 
    \begin{align}
    \label{Eq for span}
        D_i \bigcap \ker(E_i) \subset \Span\{ au_\phi \; | \; a \in A, \phi \in [[\cR_i]], \phi \cap X = \varphi \}
    \end{align}
    Now for $1 \leq j \leq n$, let $i_j \in \{1,2\}$ such that $i_j \neq i_{j+1}$ for all $1 \leq j \leq n-1$. Let $x_j \in D_{i_j} \cap \ker(E_{i_j})$ and consider the word $x_1x_2...x_n$. We claim that $E(x_1x_2....x_n) = 0$. Notice that by Equation \ref{Eq for span} and multilinearity, it is enough to show this when $x_j = a_ju_{\phi_j}$ where $a_j \in A$, $\phi_j \in [[\cR_{i_j}]]$ and $\phi_j \cap X = \varphi$. By Equation \ref{Eq: crossed product 2} we get that $x_1x_2...x_n = a(u_{\phi_1}u_{\phi_2}...u_{\phi_n}) = au_{\phi_1\phi_2...\phi_n}$ for some element $a \in A$. By the free independence of $\cR_1$ and $\cR_2$, we obtain that $\mu(\phi_1\phi_2...\phi_n \bigcap X) = 0$. Therefore, by Equation \ref{Eq: crossed product 4}, we get that: 
    \begin{align*}
        E(au_{\phi_1\phi_2...\phi_n}) = a1_{\phi_1\phi_2...\phi_n \bigcap X} = a \cdot 0 = 0
    \end{align*}
    Now consider a word $x_1x_2...x_n$ where $x_j \in M_{i_j} \cap \ker(E_{i_j})$ for each $j$. Since $D_{i_j}$ is strongly dense in $M_{i_j}$, we get a uniformly bounded sequence of elements $(z_j^m)_{m \in \N}$ by Kaplansky density theorem such that $z_j^m \rightarrow x_j$ *-strongly. Therefore the convergence happens in particular in the $L^2$-norm. Now let $y^m_j = z^m_j - E_{i_j}(z^m_j)$ for each $m$ and $j$ and notice that this still gives a uniformly bounded sequence. Moreover $E_{i_j}(y^m_j) = 0$. By the previous paragraph, we get that $E(y^m_1 y^m_2...y^m_n) = 0$. By uniform boundedness and the telescopic identity we get that $y_1^my_2^m...y_n^m \rightarrow x_1x_2...x_n$ in $L^2$-norm. Since the conditional expectation $E$ is $L^2$-contractive, this givees that $E(x_1x_2...x_n) = 0$. Thus $M_1$ and $M_2$ are freely independent over $A$ as required. 
\end{proof}

\begin{prop}
\label{prop: actions of eq relations CE}
    Let $(A_x)_{x \in X}$ be a measured field of finite hyperfinite von Neumann algebras over a standard probability space $(X,\mu)$, and let $A$ be the direct integral. Let $\cR$ be a treeable p.m.p. equivalence relation on $(X,\mu)$, let $\delta: \cR \actson A$ be an action, and let $M = A \rtimes_{\delta} \cR$ be the crossed product. Then $M$ is Connes-embeddable. 
\end{prop}
\begin{proof}
    Let $\{\phi_n \; | \; n \in \N\}$ be a countable treeing of $\cR$ and let $\cR_k$ be the subequivalence relation of $\cR$ generated by $\phi_1,\phi_2,...,\phi_k$. Let $M_k = A \rtimes_{\delta} \cR_k$ and then $M$ is the WOT-closure of the union $\bigcup_{k} M_k$. By \cite[Proposition 4.1(iii)]{OzawaSurvey}, it is enough to prove that each $M_k$ is Connes-embeddable. Let $\cS_i < \cR$ be the equivalence relation generated by the bisection $\phi_i$, and let $B_i = A \rtimes \S_i$. Since $\cS_i$ is amenable and $A$ is injective, the crossed product $B_i$ is injective, following exactly the same proof as in \cite[Proposition 4.12]{Chak25}. Since $\{\phi_n \; | \; n \in \N\}$ is a treeing, by \cite[Proposition 2.4]{AlvarezGaboriau}, $\cR_k = \cS_1 \ast \cS_2 \ast ... \ast \cS_k$ for each $k$. Then by Lemma \ref{Lemma: free products and actions} we get that $M_k = B_1 \ast_A B_2 \ast_A ... \ast_A B_k$. The result now follows from \cite[Corollary 4.5]{Brown_Dykema_Jung} as $A$ is hyperfinite. 
\end{proof}

\begin{thm}
\label{Thm: action with treeable eq relation and amenable stabilizer}
    Let $G \actson (X,\mu)$ be a p.m.p. action and let $\cR$ be the orbit equivalence relation. Suppose that $\cR$ is treeable and that the stabilizers $G_x$ are amenable for a.e. $x \in X$. Then $G$ is cocycle-hyperlinear. 
\end{thm}
\begin{proof}
    Let $\omega \in \rmZ^2(G,\T)$ and notice that since $L_{\omega}(G)$ embeds into the twisted crossed product $M = L^{\infty}(X) \rtimes_{\omega} G$, it is enough to show that the latter is Connes-embeddable. Since $\cR$ is treeable, by Lemma \ref{Lemma: twisted transformation groupoid vna}, we have that $M$ is isomorphic to $A \rtimes_\delta \cR$ where $A = \int^{\oplus}_X A_xd\mu(x)$ and $A_x = L_{\omega_x}(G_x)$. Since a.e. $G_x$ is amenable, $L_{\omega_x}(G_x)$ is hyperfinite for a.e. $x \in X$ and consequently, $A$ is hyperfinite. The result now follows from Proposition \ref{prop: actions of eq relations CE} 
\end{proof}

\section{Linear groups}
\label{Sec: Special Linear groups}

Recall that $\SL(2,\Z)$ is treeable and hence cocycle-hyperlinear. The goal of this section is to prove that $\SL(n,Z)$ is cocycle-hyperlinear for all $n \geq 3$. For a countable discrete group $G$, we note that Moore's measurable cohomology is actually the same as group cohomology. By the main result in \cite{VanderKallen}, we get that $\rmH_2(\SL(3,\Z), \Z) = \rmH_2(\SL(4,\Z), \Z) = (\Z/2\Z)^2$. It is also known (for example, see \cite[Remark after Theorem 5.10 and Theoerm 10.1]{MilnorBook}) that $\rmH_2(\SL(n,\Z)) = \Z/2\Z$ for $n \geq 5$. Therefore, $\rmH^2(\SL(n,Z), \T) = (\Z/2\Z)^2$ for $n = 3,4$ and $\rmH^2(\SL(n,Z),\T) = \Z/2\Z$ for $n \geq 5$. In fact, van der Kallen's work \cite{VanderKallen} almost immediately yields the following proposition. 

\begin{prop}
    \label{Prop: mod-4 Schur multipliers}
    Let $n \geq 3$. Then the reduction $r_4: \SL(n,\Z) \rightarrow \SL(n,\Z/4\Z)$ induces an isomorphism: 
    \begin{align*}
        (r_4)_* : \rmH_2(\SL(n,\Z),\Z) \rightarrow \rmH_2(\SL(n, \Z/4\Z), \Z)
    \end{align*}    
\end{prop}
\begin{proof}
    Let $G_n = \SL(n,\Z)$ and $Q_n = \SL(n,\Z/4\Z)$ for convenience of notation. First, assume that $n \geq 5$; we then apply Proposition \ref{Prop: random facts from Weibel's book} to $A = \Z$ and $A = \Z/4\Z$. Since their Krull dimensions are 1 and 0, and since $E_n(A) = \St_n(A)$, we get: 
    \begin{align*}
        \rmH_2(G_n,\Z) \cong \rmK_2(\Z) \cong \Z/2\Z \text{ and } \rmH_2(Q_n, \Z) \cong \rmK_2(\Z/4\Z) \cong \Z/2\Z
    \end{align*}
 Moreover, the reduction map carries the generator $z_{\Z}$ to $z_{\Z/4\Z}$ (recall the notation from Proposition \ref{Prop: random facts from Weibel's book}), so indeed $(r_4)_*$ is an isomorphism. 

    Now let $n = 3,4$. Since both $G_n$ and $Q_n$ are perfect groups, let $p_n: \widehat{G_n} \rightarrow G_n$ and $q_n: \widehat{Q_n} \rightarrow Q_n$ be their universal extensions. The reduction $r_4: G_n \rightarrow Q_n$ thus induces a homomorphism: 
    \begin{align*}
        \widehat{r_4}: \widehat{G_n} \rightarrow \widehat{Q_n}
    \end{align*}
    such that the restriction to the kernel precisely gives $(r_4)_* : \rmH_2(G_n,\Z) \rightarrow \rmH_2(Q_n,\Z)$. We will recall the construction of a canonical lift $x_{ij} \in \widehat{G_n}$ of $e_{ij}(1)$ as constructed in \cite{VanderKallen}. Indeed, let $k$ be the smallest integer in $\{1,2,...,n\}$ different from $i$ and $j$. Then let $a \in p^{-1}(\{e_{ik}(1)\})$ and let $b \in p^{-1}(\{e_{kj}(1)\})$ be arbitrary lifts and define $x_{ij} = [a,b]$. Then $x_{ij}$ does not depend on the choices $a$ and $b$ since any other lift would be of the form $ca$ and $db$ with $c,d \in \ker(p_n)$ and $[ca,db] = [a,b]$ since $\ker(p_n)$ is central. Moreover, $p_n(x_{ij}) = [e_{ij}(1),e_{ij}(1)] = e_{ij}(1)$, and hence it is indeed a lift. Now we define the following elements in $\widehat{G_n}$: 
    \begin{align*}
        w_{ij} = x_{ij}x_{ji}^{-1}x_{ij}  &\text{ and }  z_n = (w_{12})^4 \text{ for } n =3,4  \\  \eta_3 = w_{12}w_{21} &\text{ and } \eta_4 = [x_{12},x_{34}]
    \end{align*}
    By \cite[Theorem, Page 47]{VanderKallen}, $\rmH_2(G_n,\Z) = \langle \eta_n,z_n \rangle \cong (\Z/2\Z)^2$ where $\eta_n$ and $z_n$ are distinct elements of order 2. Now let $\overline{\eta}_n = \widehat{r_4}(\eta_n)$, $\overline{z}_n = \widehat{r_4}(z_n)$ and $\overline{w}_n = \widehat{r_4}(w_n)$. By \cite[Corollary, Page 48]{VanderKallen}, $\rmH_2(Q_n,\Z) = (\Z/2\Z)^2$ and hence it is enough to show that $\overline{\eta}_n$ and $\overline{z}_n$ are still non-zero and independent in $\widehat{Q_n}$. In fact it follows from the proof of \cite[Corollay, Page 48]{VanderKallen} that $\overline{\eta}_n \neq 1$ and $\widehat{Q_n}/ \langle \overline{\eta}_n \rangle = \St_n(\Z/4\Z)$ as extensions of $Q_n$. Passing to the kernels, one gets: 
    \begin{align*}
        \rmH_2(Q_n,\Z)/ \langle \overline{\eta}_n \rangle \cong \rmK_2(n, \Z/4\Z)
    \end{align*}
    Since $\Z/4\Z$ has Krull dimension 0, by Proposition \ref{Prop: random facts from Weibel's book}(2), we get that $\rmK_2(n, \Z/4\Z) \cong \rmK_2(\Z/4\Z)$. Moreover, by Proposition \ref{Prop: random facts from Weibel's book}(3), $\rmK_2(n,\Z/4\Z)$ is generated by the image of $\overline{z}_n = (\overline{w}_n)^4$. In particular, $\overline{z}_n$ has a non-trivial image in $\rmH_2(Q_n,\Z)/ \langle \overline{\eta}_n \rangle$. Therefore, $\overline{z}_n \notin \langle \overline{\eta}_n \rangle$, and $\langle \overline{\eta}_n, \overline{z}_n \rangle \cong (\Z/2\Z)^2$ as required.  
\end{proof}

Since $\Hom(\rmH_2(G,\Z),\T) \cong \rmH^2(G, \T)$ for a perfect group, consider the following group homomorphism:
\begin{align*}
    (r_4)^* : \rmH^2(\SL(n,\Z/4\Z),\T) \rightarrow \rmH^2(\SL(n,\Z), \T)
\end{align*} be 
given by mapping a homomorphism $\phi \in \Hom(\rmH_2(\SL(n,\Z),\Z),\T)$ to $\phi \circ (\r_4)_*^{-1}$. From Proposition \ref{Prop: mod-4 Schur multipliers}, it follows immediately that:

\begin{cor}
    \label{Cor: Same cohomologies for Z and Z/4Z}
    For $n \geq 3$, $(r_4)^* : \rmH^2(\SL(n,\Z/4\Z), \T) \rightarrow \rmH^2(\SL(n,\Z),\T)$ is an isomorphism. Thus, every 2-cocycle $\omega$ on $\SL(n,\Z)$ is cohomologous to a 2-cocycle of the form $\omega_0 \circ r_4$, where $\omega_0$ is a 2-cocycle on $\SL(n, \Z/4\Z)$. 
\end{cor}

We need the following lemma to complete the proof: 

\begin{lem}
\label{Lemma: finite descent}
    Let $G$ be a countable discrete group, and let $q:G \rightarrow F$ be a finite quotient. Let $\omega_0 \in \rmH^2(F,\T)$ and let $\omega \in \rmH^2(G,\T)$ be defined by $\omega(g,h) = \omega_0(q(g),q(h))$. If $G$ is hyperlinear, then $L_{\omega}(G)$ is Connes-embeddable.  
\end{lem}
\begin{proof}
    We only give a sketch as the details are routine. Let $\{u_g\}_{g \in G}$ and $\{v_f\}_{f \in F}$ denote the canonical unitaries in $L(G)$ and $L_{\omega_0}(F)$ respectively. For $g \in G$ consider the unitaries $W_g \in L_{\omega_0}(F) \otimes L(G)$ given by $W_g = v_{q(g)} \otimes u_g$. Notice that: 
    \begin{align*}
        W_g W_h = v_{q(g)}v_{q(h)} \otimes u_g u_h = \omega_0(q(g),q(h)) v_{q(gh)} \otimes u_{gh} = \omega(g,h) W_{gh}
    \end{align*}
    Considering the usual product trace, $u_g \mapsto W_g$ gives a trace preserving embedding $L_\omega(G) \hookrightarrow L_{\omega_0}(F) \otimes L(G)$. Now $L_{\omega_0}(F)$ is finite dimensional and hence Connes-embeddable and $L(G)$ is Connes-embeddable by hypothesis, thus proving the lemma. 
\end{proof}

\begin{thm}
\label{Thm: SL(n,Z) is cocycle-hyperlinear}
    For $n \geq 2$, $\SL(n,\Z)$ is cocycle-hyperlinear.   
\end{thm}
\begin{proof}
    It is well-known that $\SL(2,\Z)$ is treeable and hence cocycle-hyperlinear by Proposition \ref{Prop: treeable groups are SH}. For $n \geq 3$, since $\SL(n,\Z)$ is residually finite and hence hyperlinear, the conclusion follows immediately from Corollary \ref{Cor: Same cohomologies for Z and Z/4Z} and Lemma \ref{Lemma: finite descent}. 
\end{proof}

We record the following corollary here, although
we need a couple of permanence properties from the next section. 

\begin{cor}
    \label{Cor: GLn and PSLn}
    For $n \geq 2$, all central extensions of $\SL(n,\Z)$, $\GL(n,\Z)$ and $\PSL(n,\Z)$ are cocycle-hyperlinear. 
\end{cor}
\begin{proof}
    Notice that $\GL(n,\Z)$ is a finite index overgroup of $\SL(n,\Z)$ and hence cocycle-hyperlinear by Theorem \ref{Thm: SL(n,Z) is cocycle-hyperlinear} and Proposition \ref{Prop: coamenable overgroups}. Similarly, $\PSL(n,\Z)$ is a quotient of $\SL(n,\Z)$ by $Z(\SL(n,\Z))$ and and hence cocycle-hyperlinear by Theorem \ref{Thm: SL(n,Z) is cocycle-hyperlinear} and Proposition \ref{Prop: quotient by central subgroup}. For central extensions of $\SL(n,\Z)$, it follows from \ref{Prop: extensions of perfect groups}. Now for a central extension of $\GL(n,\Z)$ given by:
    \begin{align*}
        1 \rightarrow Z \rightarrow G \xrightarrow{p} \GL(n,\Z) \rightarrow 1
    \end{align*}
    let $G_0 = p^{-1}(\SL(n,\Z))$. Then: 
    \begin{align*}
        1 \rightarrow Z \rightarrow G_0 \xrightarrow{p} \SL(n,\Z) \rightarrow 1
    \end{align*}
    is a central extension of $\SL(n,\Z)$. By what we already concluded, $G_0$ is cocycle-hyperlinear. Since $G_0$ is an index-2 subgroup of $G$, we get that $G$ is cocycle-hyperlinear by Proposition \ref{Prop: coamenable overgroups}. Now consider a central extension 
    \begin{align*}
        1 \rightarrow Z \rightarrow G \xrightarrow{p} \PSL(n,\Z) \rightarrow 1
    \end{align*}
    Let $q: \SL(n,\Z) \rightarrow \PSL(n,\Z)$ be the quotient map and let $\widetilde{G} = \{(h,g) \in \SL(n,\Z) \times G \; | \; q(h) = p(e)\}$. Projecting onto the first coordinate gives a central extension
    \begin{align*}
        1 \rightarrow Z \rightarrow \widetilde{G} \xrightarrow{p} \SL(n,\Z) \rightarrow 1
    \end{align*}
    Then $\widetilde{G}$ is cocycle-hyperlinear and the projection onto the second coordinate G has kernel $Z(\SL(n,\Z)) \times \id$. Thus once again the result follows from Proposition \ref{Prop: quotient by central subgroup}.    
\end{proof}

\section{Other permanence properties}

\subsection{Direct products and directed unions}

\begin{prop}
\label{Prop: direct product of SH is in SH}
    Suppose that $G$ and $H$ are cocycle-hyperlinear, then $G \times H$ is cocycle-hyperlinear. 
\end{prop}
Before proving this proposition we need the usual Kunneth formula for group cohomology which is known to experts. However, we add a proof for completeness. Recall that a $G-H$ \textit{bicharacter} $b: G \times H \rightarrow \T$ is a Borel map satisfying $b(gg',h) = b(g,h)b(g',h)$ and $b(g,hh')= b(g,h)b(g,h')$ for all $g,g' \in G$ and $h,h' \in H$.

\begin{lem}
\label{Lemma: cocycle breaks for direct products}
    Let $\omega \in \rmZ^2(G \times H, \T)$ be a normalized cocycle, then $\omega$ is cohomologous to a cocycle of the form $\omega_G \times \omega_H \times b : G \times H \rightarrow \T$ where $\omega_G \in \rmZ^2(G,\T)$, $\omega_H \in \rmZ^2(H,\T)$ and $b$ is a $G-H$ bicharacter. 
\end{lem}
\begin{proof}
    Let $\omega_G(g,g') = \omega((g,e),(g',e))$ and $\omega_H(h,h') = \omega((e,h),(e,h'))$, then $\omega_G$ and $\omega_H$ are clearly normalized 2-cocycles on the respective groups. Define the following Borel map $b: G \times H \rightarrow \T$ by: 
    \begin{align*}
        b(g,h) = \frac{\omega((g,e),(e,h))}{\omega((e,h),(g,e))}
    \end{align*}
    We claim that $b$ is a bicharacter. Indeed notice that by the cocycle identity we get: 
    \begin{align}
    \label{Eq for bicharacter G}
        \omega((g,e),(g',e))\omega((gg',e),(e,h)) &= \omega((g',e),(e,h))\omega((g,e),(g',h)) \text{ and }  \\ \label{eq for bicharacter H}
        \omega((e,h),(g,e))\omega((g,h),(g',e)) &= \omega((g,e),(g',e))\omega((e,h),(gg',e)) \\  \label{third eq for bicharacter}
        \omega((g,e),(e,h))\omega((g,h),(g',e)) &= \omega((e,h),(g',e))\omega((g,e),(g',h))
    \end{align}
    Now we calculate the following from Equations \ref{Eq for bicharacter G} and \ref{eq for bicharacter H}:
    \begin{align*}
        b(gg',h) = \frac{\omega((gg',e),(e,h))}{\omega((e,h),(gg',e))} &= \frac{\omega((g',e),(e,h))\omega((g,e),(g',h))\omega((g,e),(g',e))}{\omega((g,e),(g',e))\omega((e,h),(g,e))\omega((g,h),(g',e))} \\ &= \frac{\omega((g',e),(e,h))\omega((g,e),(g',h))}{\omega((e,h),(g,e))\omega((g,h),(g',e))} \\ \text{ By equation \ref{third eq for bicharacter} } &= \frac{\omega((g',e),(e,h))\omega((g,e),(e,h))}{\omega((e,h),(g,e))\omega((e,h),(g',e))} \\ &= b(g,h) b(g',h)
    \end{align*}
    By a completely symmetric argument we have $b(g,hh') = b(g,h)b(g,h')$ and hence $b$ is a bicharacter. Now consider a 1-cochain $\eta$ given by: $\eta(g,h) = \omega((g,e),(e,h))$, and note that $\eta$ is still normalized because $\omega$ is normalized. Again consider the following cocycle identities: 
    \begin{align}
    \label{eq 4 bicharacter}
        \omega((g,e),(e,h))\omega((g,h),(g',h')) &= \omega((g,e),(g',hh'))\omega((e,h),(g',h')) \\
        \label{eq 5 bicharacter}
       \omega((g,e),(g',e))\omega((gg',e),(e,hh')) &= \omega((g',e),(e,hh'))\omega((g,e),(g',hh'))\\
        \label{eq 6 bicharacter}
       \omega((e,h),(g',e))\omega((g',h),(e,h')) &= \omega((g',e),(e,h'))\omega((e,h),(g',h'))
       \\
        \label{eq 7 bicharacter}
       \omega((g',e),(e,h))\omega((g',h),(e,h')) &= \omega((e,h),(e,h'))\omega((g',e),(e,hh'))
    \end{align}

    Now we have:
    \begin{align*}
&\eta(g,h)\eta(g',h')\overline{\eta(gg',hh')}\omega((g,h),(g',h')) \\ =\omega((g,e),(e,h))&\omega((g',e),(e,h'))\overline{\omega((gg',e),(e,hh'))}\omega((g,h),(g',h')) \\ \text{ by } \ref{eq 4 bicharacter} = \omega((g, e), &(g',hh'))\omega((e,h),(g',h')) \omega((g',e),(e,h'))\overline{\omega((gg',e),(e,hh'))} \\ \text{ by } \ref{eq 5 bicharacter} = \omega((g, e), &(g',e))\omega((g',e),(e,h')) \omega((e,h),(g',h'))\overline{\omega((g',e),(e,hh'))}\\ \text{ by } \ref{eq 6 bicharacter} = \omega((g, e), &(g',e))\omega((e,h),(g',e)) \omega((g',h),(e,h'))\overline{\omega((g',e),(e,hh'))} \\ \text{ by } \ref{eq 7 bicharacter} = \omega((g, e), &(g',e))\omega((e,h),(e,h')) \omega((e,h),(g',e))\overline{\omega((g',e),(e,h))} \\ & = \omega_G(g,g')\omega_H(h,h')\overline{b(g',h)} 
    \end{align*}
Hence $\omega$ is cohomologous to a product of the restricted cocycles with a bicharacter as required. 
\end{proof}

\begin{proof}[Proof of Proposition \ref{Prop: direct product of SH is in SH}]
    Let $\omega$ be a 2-cocycle on $G \times H$. By Lemma \ref{Lemma: cocycle breaks for direct products} we can assume that $\omega = \omega_G \times \omega_H \times b$. Let $A = G_{\ab}$ and $B = \rmH_{\ab}$ be the abelianizations and notice that a bicharacter factors through the abelianizations, and hence gives a bicharacter $\beta : A \times B \rightarrow \T$. Consider the 2-cocycle $\rho$ on $A \times B$ given by: 
    \begin{align*}
        \rho((a,b),(a',b')) = \overline{\beta(a',b)}
    \end{align*}
    One checks that this is indeed a 2-cocycle because: 
    \begin{align*}
        &\rho((a,b),(a',b'))\rho((aa',bb'),(a'', b'' )) = \overline{\beta(a',b)\beta(a'' , bb')} \\ & = \overline{\beta(a',b)\beta(a'',b) \beta(a'',b')} = \overline{\beta(a'a'', b)\beta(a'' ,b')} \\ &= \rho((a,b),(a'a'',b'b'')) \rho ((a',b'), (a'',b''))
    \end{align*}
    Let $M_\beta = L_\rho(A \times B)$. Since $A \times B$ is abelian, by Proposition \ref{Prop: treeable groups are SH}, $M_\beta$ is Connes-embeddable. Let $(u_a)_{a \in A}$ and $(v_b)_{b \in B}$ be the canonical unitaries in $M_\beta$ and notice that they satisfy: 
    \begin{align*}
        u_av_b = \beta(a,b)v_b u_a \text{ in } M_\beta
    \end{align*}
    Let $M_G = L_{\omega_G}(G)$ and $M_H = L_{\omega_H}(H)$ and
    let $(x_g)_{g \in G}$ and $(y_h)_{h \in H}$ denote the canonical group unitaries in $M_G$ and $M_H$. Let $(z_{(g,h)})_{(g,h) \in G \times H}$ denote the canonical unitaries in $L_\omega(G \times H)$. Consider the map $\theta$ from the canonical unitaries in $L_\omega(G \times H)$ to $M_G \otimes M_H \otimes M_\beta$ given by: 
    \begin{align*}
        \theta(z_{(g,h)}) = x_g \otimes y_h \otimes u_{[g]}v_{[h]}
    \end{align*}
    where $[g]$ and $[h]$ denote their images in the abelianizaitons. One checks that $(u_{[g]} v_{[h]})(u_{[g']}v_{[h']}) = \overline{b(g',h)}u_{[gg']}v_{[hh']}$ and consequently:
    \begin{align*}
        \theta(z_{(g,h)})\theta(z_{(g', h')}) &= (x_g \otimes y_h \otimes u_{[g]}v_{[h]})(x_{g'} \otimes y_{h'} \otimes u_{[g']}v_{[h']}) \\ &= \omega_G(g,g') x_{gg'} \otimes \omega_H(h,h') y_{hh'} \otimes \overline{\beta(g',h)} u_{[gg']}v_{[hh']} \\ &= \omega((g,h)(g',h'))\theta(z_{(gg',hh')}) \\&= \theta(z_{(g,h)}z_{(g',h')})
    \end{align*}
    Thus $\theta$ extends to a *-homomorphism to the algebraic twisted group algebra. Notice that $\tau(\theta(z_{(g,h)})) = \tau_{M_G}(x_g)\tau_{M_H}(y_h) \tau_{M_\beta}(u_{[g]}v_{[h]}) = \delta_{(g,h),(e,e)}$ and hence $\theta$ preserves the trace on the canonical unitaries. Thus $\theta$ extends to a normal trace preserving embedding:
    \begin{align*}
        \theta: L_{\omega}(G \times H) \rightarrow L_{\omega_G}(G) \otimes L_{\omega_H}(H) \otimes L_{\rho}(A \times B)
    \end{align*}
    Since $G$ and $H$ are cocycle-hyperlinear, all three von Neumann algebras on the right in the above equation are Connes-embeddable and hence $L_\omega(G \times H)$ is Connes-embeddable. 
\end{proof}

\begin{ex}
    There are many examples of products of treeable groups that are not treeable. If $G$ and $H$ are infinite and at least one of them contain one infinite order element, then the cost of any essentially free p.m.p. action of $G \times H$ is 1. But if a group admits a free p.m.p. action with a treeing of cost 1, it is amenable. So whenever at least one of $G$ and $H$ is non-amenable we have that $G \times H$ is not treeable. We refer the reader to \cite{Gaboriau_ICM} for a survey of these results. In particular $\mathbb{F}_m \times \mathbb{F}_n$ is not treeable but is cocycle-hyperlinear whenever at least one of $m$ or $n$ is at least 2.    
\end{ex}

\begin{lem}
\label{Lemma: directed union}
    Let $G_1 \leq G_2 \leq...$ be an increasing sequence of groups that are cocycle-hyperlinear. Then the directed union $G = \bigcup_n G_n$ is cocycle-hyperlinear. 
\end{lem}
\begin{proof}
    It is an easy check that if $\omega \in \rmH^2(G,\T)$, then $L_{\omega}(G) = \left( \bigcup_n L_\omega(G_n) \right)''$. Now the result follows from \cite[Proposition 4.1(iii)]{OzawaSurvey}. 
\end{proof}

\begin{cor}
\label{Cor: countable direct sum}
    Let $G_n$ be a cocycle-hyperlinear group for all $n \in \N$. Then the countable direct sum $G = \bigoplus G_n$ is cocycle-hyperlinear.
\end{cor}
\begin{proof}
    This is immediate from Proposition \ref{Prop: direct product of SH is in SH} and Lemma \ref{Lemma: directed union}. 
\end{proof}

Let us call a group \textit{locally cocycle-hyperlinear} if every finitely generated subgroup of $G$ is cocycle-hyperlinear. 

\begin{cor}
    If $G$ is locally cocycle-hyperlinear, then it is cocycle-hyperlinear
\end{cor}
\begin{proof}
    Enumerate $G = \{g_1,g_2,g_3,..\}$ and let $\rmH_n$ be the group generated by $\{g_1,...,g_n\}$. By assumption, $\rmH_n$ is cocycle-hyperlinear and by Lemma \ref{Lemma: directed union}, $G$ is cocycle-hyperlinear.  
\end{proof}

\subsection{Quotients and overgroups}

In this section we show that certain quotients of cocycle-hyperlinear groups remain cocycle-hyperlinear. 

\begin{prop}
\label{Prop: quotients by finite subgroups}
    Let $N < G$ be a finite normal subgroup. If $G$ is cocycle-hyperlinear then $G/N$ is cocycle-hyperlinear. 
\end{prop}
\begin{proof}
  Let $Q = G/N$, $\pi: G \rightarrow Q$ denote the quotient map, $\omega_0 \in \rmH^2(Q, \T)$ and $\omega(g,h) \coloneqq \omega_0(\pi(g),\pi(h))$ denote the induced cocycle in $\rmH^2(G,\T)$. Let $\{u_g\}$ denote the canonical unitaries in $L_\omega(G)$, and notice that $L_\omega(G)$ is Connes-embeddable by hypothesis. Let $p = \frac{1}{|N|} \sum_{n \in N} u_n$. Notice that since $\omega(m,n) = 1$ for all $m,n \in N$: 
    \begin{align*}
        p^2 = \frac{1}{|N|^2}\sum_{m,n} u_m u_n = \frac{1}{|N|^2}\sum_{m,n} \omega(m,n) u_{m,n} = \frac{1}{|N|^2}\sum_{m,n} u_{mn} = \frac{1}{|N|}\sum_{m} u_{m} = p
    \end{align*}
    Since $p$ is clearly self adjoint, $p$ is a projection. Moreover we Calculate: 
    \begin{align*}
        u_g u_n &= \omega(g,n) u_{gn} = \omega_0(\pi(g),1)u_{gn} = u_{gn} = u_{gng^{-1}g} = u_{gng^{-1}}u_g
    \end{align*}
    for all $g \in G$ and $n \in N$. Since $N$ is normal, conjugation by $g \in G$ permutes the elements in $N$ and hence $p$ commutes with $u_g$ for all $g \in G$, thus $p \in Z(L_{\omega}(G))$. Now pick a normalized section $s: Q \rightarrow G$ and for all $q \in Q$ let $v_q = pu_{s(q)}p \in pL_{\omega}(G)p$. Let $n(q,r) = s(q)s(r)s(qr)^{-1}$ for all $q,r \in Q$ and calculate: 
    \begin{align*}
        v_q v_r = pu_{s(q)}u_{s(r)}p = \omega(s(q),s(r))pu_{s(q)s(r)} = \omega_0(q,r)p u_{n(q,r)}u_{s(qr)}p = \omega_0(q,r)v_{qr}  
    \end{align*}
    The last equality above follows as $pu_n = p$ for all $n \in N$. Let $\tau$ be the normalized trace on $p L_{\omega}(G) p$ given by $\tau = |N| \tau_{L_{\omega}(G)}$ and notice that $\tau(v_q) = \delta_{q,e}$. Moreover, $p L_{\omega}(G) p $ is generated by $\{v_q \; | \; q \in Q\}$ and hence $pL_{\omega}(G)p \cong L_{\omega_0}(Q)$. Since corners of Connes-embeddable von Neumann algebras are Connes-embeddable, the result is proven. 
\end{proof}

On the other hand, if we restrict ourselves to quotients by central subgroups, then we do not require finiteness as depicted in the next proposition: 

\begin{prop}
\label{Prop: quotient by central subgroup}
    Let $N<G$ be a central normal subgroup, i.e., $N < Z(G)$. If $G$ is cocycle-hyperlinear, so is $G/N$. 
\end{prop}
\begin{proof}
    Let $Q = G/N$ as before and let $\pi: G \rightarrow Q$ be the quotient map. Let $\omega_0 \in \rmH^2(Q,\T)$ and let $D < \T$ be the countable subgroup generated by the image of $\omega_0$. Consider the $D$-valued cocycle $\omega(g,h) = \omega_0(\pi(g),\pi(h))$ for all $g,h \in G$. Consider the group $E = D \times_\omega G$ given by $(d,g)\cdot(d',g') = (dd'\omega(g,g'), gg')$. One checkes that associativity follows from the cocycle relations. The projection map $E \rightarrow G$ gives a central extension as the kernel is precisely $D$ which clearly is in the center of $E$. Now let $\chi \in \widehat{D}$ be a character and notice that for the coycles $\omega_\chi = \chi \circ \omega$, $L_{\omega_\chi}(G)$ is Connes-embeddable as $G$ is cocycle-hyperlinear. By \cite[Lemma 3.4]{Thom10}, we get that $E$ is hyperlinear. 

    Consider the surjective group homomorphism $\rho: E \rightarrow Q$ given by $\rho(d,g) = \pi(g)$.  Notice that $\ker(\rho) = \{(d,g) \in E \; | \; g \in \ker(\pi) = N\} = D \times N$. Now let $(d,n) \in D \times N$ and notice that $(d,n)(d',g) = (dd'\omega(g,n), ng) = (d'd\omega_0(\pi(g),1),gn) = (d',g)(d,n)$. Thus $D \times N \subset Z(E)$ and we have a central extension $1 \rightarrow D \times N \rightarrow E \rightarrow Q \rightarrow 1$. Pick a normalized section $s: Q \rightarrow G$ for $\pi$ and let $n(q,r) = s(q)s(r)s(qr)^{-1}$ for all $q,r \in Q$. Consider the corresponding normalized section $s': Q \rightarrow E$ given by $s'(q) = (1, s(q))$. The cocycle associated to the extension is given by: 
    \begin{align*}
        \alpha(q,r) = (\omega_0(q,r),n(q,r)) \in D \times N
    \end{align*}
    Now consider the character $\kappa \in \widehat{D \times N}$ given by $\kappa(d,n) = d$ and notice that $\kappa \circ \alpha = \omega_0$. Since $E$ is hyperlinear, once again by \cite[Lemma 3.4]{Thom10}, we get that $L_{\omega}(Q)$ is hyperlinear.
\end{proof}

\begin{cor}
\label{Cor: central products}
    Let $G$ and $H$ be cocycle-hyperlinear groups and let $K < Z(G)$ and $L < Z(H)$ be central subgroups and let $\theta: K \rightarrow L$ be an isomorphism. Then the central product $(G \times H ) / \{(k,\theta(k)^{-1}) \; | \; k \in K\}$ is cocycle-hyperlinear.   
\end{cor}
\begin{proof}
    Notice that $G \times H$ is cocycle-hyperlinear by Proposition \ref{Prop: direct product of SH is in SH}. Since the subgroup that we quotient with is central, the result follows from Proposition \ref{Prop: quotient by central subgroup}. 
\end{proof}

The following proposition is an application of the previous results on quotients and can be used to see that even central extensions of $\SL(n,\Z)$ for $n \geq 5$ are cocycle-hyperlinear. 

\begin{prop}
    \label{Prop: extensions of perfect groups}
    Let $Q$ be a perfect cocycle-hyperlinear group and let $1 \rightarrow K \rightarrow G \xrightarrow{q} Q \rightarrow 1$ be a central extension. Then $G$ is cocycle-hyperlinear.
\end{prop}
\begin{proof}
    Since $Q$ is perfect, it admits a universal central extension:
    \begin{align*}
        1 \rightarrow \rmH_2(Q,\Z) \rightarrow \widetilde{Q} \xrightarrow{p} Q \rightarrow 1
    \end{align*}
    In this situation it turns out that $\rmH_1(\widetilde{Q},\Z) = \rmH_2(\widetilde{Q}, \Z) = \{0\}$ (for example by \cite[Proposition 2.3]{Bridson20}). Since $Q$ is cocycle-hyperlinear, by Thom's criterion \cite[Lemma 3.4]{Thom10}, $\widetilde{Q}$ is hyperlinear. Notice that applying the universal coefficient theorem to $\widetilde{Q}$, since $\rmH_1$ and $\rmH_2$ vanish, we get that $\rmH^2(\widetilde{Q},\T) = \{0\}$. Thus every twisted group von Neumann algebra of $\widetilde{Q}$ is isomorphic to $L(\widetilde{Q})$ and hence we get that $\widetilde{Q}$ is cocycle-hyperlinear. By universality, we have a unique group homomorphism $\phi: \widetilde{Q} \rightarrow G$ such that $q \circ \phi = p$. Define the map $\theta: \widetilde{Q} \times K \rightarrow G$ given by $\theta(u,k) = \phi(u)k$. Notice that since $K < Z(G)$, the map $\theta$ is actually a group homomorphism. Moroever it is surjective: indeed for $g \in G$, pick $u \in \widetilde{Q}$ such that $q(g) = p(u)$ in $Q$; then $q(g\phi(u)^{-1}) = q(g)p(u)^{-1} = e$ and hence $g\phi(u)^{-1} \in K$. Then $\theta(u, g \phi(u)^{-1}) = \phi(u)g\phi(u)^{-1} = g$ as required. 

    We claim now that $\ker(\theta) < Z(\widetilde{Q} \times K) = Z(\widetilde{Q}) \times K$. Indeed notice that if $\phi(u)k = e$ for $u \in \widetilde{Q}$ and $k \in K$, then in particular $\phi(u) \in K$ and $q(\phi(u)) = p(u) = e$. Thus $u \in \ker(p) \leq Z(\widetilde{Q})$ as required. Now $K$ is abelian, hence cocycle-hyperlinear and by Proposition \ref{Prop: direct product of SH is in SH}, $\widetilde{Q} \times K$ is  cocycle-hyperlinear. By the first isomorphism theorem applied to $\theta: \widetilde{Q} \times K \rightarrow G$, we have that $G \cong \widetilde{Q} \times K/ \ker(\phi)$. Since $\ker(\phi)$ is central, by Proposition \ref{Prop: quotient by central subgroup}, $G$ is cocycle-hyperlinear. 
\end{proof}

Recall that a subgroup $H \leq G$ is called \textit{co-amenable} if the action $G \actson G/H$ admits an invariant mean. If $H$ is normal, then co-amenability is equivalent to $G/H$ being amenable. The main result of this section is the following:

\begin{prop}
\label{Prop: coamenable overgroups}
    Let $H \leq G$ be a co-amenable subgroup. If $H$ is cocycle-hyperlinear, then $G$ is cocycle-hyperlinear. In particular if $G$ has a finite index cocycle-hyperlinear subgroup, then $G$ is cocycle-hyperlinear.
\end{prop}
\begin{proof}
    Let $\omega \in \rmH^2(G,\T)$ and let $D \leq \T$ be the countable subgroup generated by the image of $\omega$. Exactly as in the proof of Proposition \ref{Prop: quotient by central subgroup}, consider the central extension $E_G = D \times_\omega G$ where multiplication is given by $(d,g)(d',g') = (d'd \omega(g,g'), gg')$. Let $E_H$ be the subgroup $D \times_\omega H$ and consider the central extension $1 \rightarrow D \rightarrow E_H \rightarrow H \rightarrow 1$. Since $H$ is cocycle-hyperlinear, by \cite[Lemma 3.4]{Thom10}, $E_H$ is hyperlinear. Now notice that the map $(d,g)E_H \mapsto gH$ gives a $E_G$ equivariant bijection $E_G/E_H \rightarrow G/H$. Thus $E_H$ is co-amenable in $E_G$. Since $E_H$ is hyperlinear, by \cite[Proposition 5.1]{Brude_Sasyk}, $E_G$ is hyperlinear. Consider now the central extension $1 \rightarrow D \rightarrow E_G \rightarrow G \rightarrow 1$ and notice that since the inclusion map $\iota : D \rightarrow \T$ is a character, we get that $\iota \circ \omega = \omega$ and once again by Thom's criterion \cite[Lemma 3.4]{Thom10} we get that $L_{\omega}(G)$ is Connes-embeddable, as required. The last sentence in the proposition follows as every finite index subgroup is co-amenable.  
\end{proof}

\begin{cor}
\label{Cor: SH by amenable groups}
    Let $1 \rightarrow K \rightarrow G \rightarrow Q \rightarrow 1$ be an exact sequence of groups. If $K$ is cocycle-hyperlinear and $Q$ is amenable then $G$ is cocycle-hyperlinear.  
\end{cor}
\begin{proof}
    Since $Q \cong G/K$ is amenable, thus $K \leq G$ is co-amenable and hence this follows from Proposition \ref{Prop: coamenable overgroups}. 
\end{proof}

Notice that in particular, Corollary \ref{Cor: SH by amenable groups} says that treeable-by-amenable groups are cocycle-hyperlinear. We have already obtained in \ref{Cor: amenable by treeable groups are SH} that amenable-by-treeable groups are cocycle-hyperlinear.

\subsection{Amalgamated free products and HNN extensions}
The proof of the following proposition is exactly similar to the proof of Lemma \ref{Lemma: free products and actions} and hence we only give a sketch. 

\begin{lem}
\label{Lemma: amalgam twisted group von neumann algebras}
    Let $H$ and $K$ be countable groups with a common subgroup $L$ and let $G = H \ast_L K$. Let $\omega \in \rmH^2(G,\T)$. Then: 
    \begin{align*}
        L_{\omega}(G) \cong L_{\omega}(H) \ast_{L_{\omega}(L)} L_{\omega}(K)
    \end{align*}
\end{lem}
\begin{proof}
    For convenience of notation let $M = L_{\omega}(G)$, $P = L_{\omega}(H)$, $Q = L_\omega(K)$, $B = L_{\omega}(L)$ and $E: M \rightarrow B$ be the unique trace preserving conditional expectation. By abusing notation we denote the restrictions $E|_{P}: P \rightarrow B$ and $E|_{Q}: Q \rightarrow B$ by $E$. Let $\{u_{g} \; | \; g \in G\}$ denote the canonical unitaries in $M$ and one has that $E(u_g) = \delta_{g,L}u_g$ where $\delta_{g,L}$ is 1 if and only if $g \in L$, otherwise 0. Now since $G$ is the amalgamated free product, every $g \in G$ is of the form $g_1g_2...g_n$ where $g_i's$ are in $H$ and $K$. By the twisted multiplication formula, one has $u_{g_1}u_{g_2}...u_{g_n} = c u_{g}$ for a scalar $c \in \T$. Since $M$ is generated by these canonical unitaries, we have that $M = W^*(P \bigcup Q)$.

    Let $P_0$ and $Q_0$ be the twisted group algebras generated by the corresponding group unitaries. Notice that $P_0 \cap \ker(E) = \Span \{u_g \; | \; g \in H \backslash L\}$ and $Q_0 \cap \ker(E) = \Span \{u_g \; | \; g \in K \backslash L\}$. Let $g = h_1\rmK_1h_2\rmK_2...h_n\rmK_n$ where $h_i \in H \backslash L$, $\rmK_i \in K \backslash L$, then by the Bass-Serre normal form theorem, $g$ is reduced and does not belong to $L$. Moreover $u_{h_1}u_{\rmK_1}...u_{h_n}u_{\rmK_n} = cu_{g}$  for a scalar $c \in \T$. Therefore $E(u_{h_1}u_{\rmK_1}...u_{h_n}u_{\rmK_n}) = 0$ from the conditional expectation stated in the previous paragraph. We can now pass from this algebraic freeness to freeness of $P$ and $Q$ over $B$ by using the exact same argument as in Lemma \ref{Lemma: free products and actions} using Kaplansky density theorem and the contractiveness of the conditional expectation.  
\end{proof}

\begin{prop}
\label{Prop: amalagamted free products over amenable groups}
    Let $G = H \ast_L K$ be an amalgamated free product. If $H$ and $K$ are cocycle-hyperlinear and $L$ is amenable, then $G$ is cocycle-hyperlinear. 
\end{prop}
\begin{proof}
   Let $\omega$ be a 2-cocycle on $G$. By Lemma \ref{Lemma: amalgam twisted group von neumann algebras}, we have $L_{\omega}(G) \cong L_{\omega}(H) \ast_{L_{\omega}(L)} L_{\omega}(K)$. Since $H$ and $K$ are cocycle-hyperlinear, the two twisted von Neumann algebras are Connes-embeddable. Since $L$ is amenable, $L_{\omega}(L)$ is hyperfinite. The result now follows from \cite[Corollary 4.5]{Brown_Dykema_Jung}. 
\end{proof}

\begin{cor}
    For $n \in \N$, let $G_n$ be a cocycle-hyperlinear group and let $L$ be a common amenable subgroup. Then the countable amalgamated free product $G = \ast_{L} G_{n}$ is cocycle-hyperlinear. 
\end{cor}
\begin{proof}
    Every finite product is cocycle-hyperlinear by Proposition \ref{Prop: amalagamted free products over amenable groups}. The countable directed union is cocycle-hyperlinear by Lemma $\ref{Lemma: directed union}$.  
\end{proof}

We recall the definition of an HNN extension. Let $H \leq G$, let $\nu: H \rightarrow G$ be an injective homomorphism and let $K = \nu(H)$. Then the HNN extension is defined as: 
\begin{align*}
    \Gamma \coloneqq \HNN(G,H,\nu) = \langle G,t \; | \; tht^{-1} = \nu(h) \text{ for all } h\in H \rangle 
\end{align*}
Now let $\omega \in \rmZ^2(\Gamma, \T)$ and for convenience let us denote $M = L_{\omega|_G}(G)$, $P = L_{\omega|_H}(H)$ and $Q = L_{\omega|_{K}}(K)$. 
Denote the group unitaries in $L_\omega(\Gamma)$ by $\{u_{g} \; | \; g \in \Gamma\}$. Notice that for the unitary $u_t$ associated to the symbol $t$, the conjugation map $\Theta_{\omega} \coloneqq \Ad(u_t)$ maps $P$ trace-preservingly and isomorphically to $Q$. One calculates that on generators, the cocycle identity forces the map to be given by:
\begin{align*}
    \Theta_\omega(u_h) = d(h) u_{\nu(h)} \text{ where } d(h) = \omega(t,h)\omega(th,t^{-1})\overline{\omega(t^{-1},t)}
\end{align*}
because $u_t^* = \overline{\omega(t^{-1},t)}u_{t^{-1}}$ and $tht^{-1} = \nu(h)$. In \cite{Ueda}, this construction is generalized for von Neumann algebras. In our specific setting, $\HNN(M,P,\Theta_\omega)$ is generated by $M$, and the \textit{stable unitary} $v$ that satisfies: 
\begin{align*}
    v x v^{*} = \Theta_\omega(x) \text{ for all } x \in P
\end{align*}
and the trace vanishes on reduced words. The identity on $M$ and the map $v \mapsto u_t$ satisfy the covariance condition. By an application of Britton's Lemma (exactly as in the untwisted case \cite[Corollary 3.5]{Ueda}, the canonical trace on $L_\omega(\Gamma)$ vanishes on the image of algebraically reduced words. By bounded approximation, the same trace condition holds for reduced words with coefficients in $M$, Consequence, we get a trace preserving embedding $\HNN(M,P,\Theta_\omega) \rightarrow M$. This gives us the following proposition:   

\begin{prop}
    \label{Prop: HNN extension isomorphism}
    Let $\Gamma = \HNN(G,H, \nu)$ and let $\omega \in \rmZ^2(\Gamma,\T)$. Then there is a trace-preserving isomorphism $L_{\omega}(\Gamma) \cong \HNN(M,P,\Theta_\omega)$ where $M = L_\omega(G)$, $P = L_\omega(H)$ and $\Theta_\omega$ is as above. 
\end{prop}

Before looking at the main application of Proposition \ref{Prop: HNN extension isomorphism} in our setting, let us briefly recall the notion of graphs of tracial von Neumann algebras from \cite[Section 6]{Fima_Amaury}. Let $\cG$ be a connected graph with vertex set $V(\cG)$, edge set $E(\cG)$ and the involution $e \rightarrow \overline{e}$ of reversing orientation. A \textit{graph of von Neumann algebras} is a tuple 
\begin{align*}
    \left( \cG, (M_q, \tau_q)_{q \in V(\cG)}, (N_e, \tau_e)_{e \in E(\cG)}, (s_e)_{e \in E(\cG)} \right)    
\end{align*}
together with the following conditions:
\begin{enumerate}
    \item For each $q \in V(\cG)$ and $e \in E(\cG)$, $M_q$ and $N_e$ are tracial von Neumann algebras with faithful tracial states $\tau_q$ and $\tau_e$ respecitively. 
    \item For all $e \in E(\cG)$, we have $N_{\overline{e}} = N_e$ and $\tau_{\overline{e}} = \tau_e$. 
    \item For every $e \in E(\cG)$, $s_e: N_{e} \rightarrow M_{s(e)}$ is a unital normal faithful trace-preserving *-homomorphisms.
\end{enumerate}
For every $e \in E(\cG)$, we denote $r_e = s_{\overline{e}}: N_e \rightarrow M_{r(e)}$. So each edge algebra $N_e$ embeds into the two vertex algebras $M_{s(e)}$ and $M_{r(e)}$, attached to the endpoints of $e$ through the maps $s_e$ and $r_e$. We shall also denote $N^s_e = s_e(N_e)$ and $N^r_e = r_e(N_e)$. Recall from \cite[Definition 6.3]{Fima_Amaury}, the definition of the fundamental von Neumann algebra associated to the graph $\pi_1(\cG)$. This construction completely recovers amalgamated free products, as well as HNN extensions \cite[Example 6.5]{Fima_Amaury}. 

Let $(P,\tau_P)$ and $(Q, \tau_Q)$ be tracial von Neumann algebras and let $B$ be a common subalgebra. 
Let $\cG$ be the graph of two vertices $p$ and $q$ and two edges $e: p \rightarrow q$ and $\overline{e}: q \rightarrow p$. Let $P = M_p$, $Q = M_q$ and $B = N_e$ be the corresponding von Neumann algebras. In this case $\pi_1(\cG) \cong P \ast_{B} Q$ (see \cite[Example 3.3]{Fima_Amaury}). Now let $B \subset M$ be a von Neumann subalgebra and let $\Theta: B \rightarrow M$ be an injective trace preserving *-homomorphism. Now let $\cG$ be the graph consisting of one vertex $p$ and two edges $e$ and $\overline{e}$ which are loops on $p$. In this case taking $M_p = M$, $N_e = B$, $s_e = \iota$ and $r_e = \theta$, it turns out that $\pi_1(\cG) \cong \HNN(M,B, \theta)$ (see \cite[Example 3.4]{Fima_Amaury}). Now we can prove the following result about stable hyperlinearity of HNN extensions: 

\begin{cor}
    \label{Prop: HNN extensions SH}
    Let $H \leq G$ where $G$ is cocycle-hyperlinear and $H$ is amenable and let $\nu: H \rightarrow G$ be injective. Then $\HNN(G,H,\nu)$ is cocycle-hyperlinear. 
\end{cor}
\begin{proof}
    Let $M = L_{\omega}(\Gamma)$, $B = L_{\omega}(G)$, $\omega \in \rmH^2(\Gamma,\T)$ and $\Theta_\omega = \Ad(u_t)$ where $t$ is the symbol generating $\Gamma$ together with $G$. 
    By Proposition \ref{Prop: HNN extension isomorphism}, $L_{\omega}(G) \cong \HNN(M,B,\Theta_\omega)$. Now consider the graph $\cG$ with $V(\cG) = \{p\}$ and $E(\cG) = \{e, \overline{e}\}$ where $M_p = M$, $N_e = B$, $r_e = \Theta_\omega$ and $s_e = \iota$. By \cite[Corollary 6.10]{Fima_Amaury}, $\pi_1(\cG)$ is Connes-embeddable if the edge algebra $B$ is amenable and the vertex algebra $M$ is Connes-embeddable. Since both of these follow from the hypothesis, we are done with the proof. (Notice that the same idea also gives a proof for the amalgamated free product case, but that was already known from \cite{Brown_Dykema_Jung}).  
\end{proof}

Recall that for non-zero integers $m,n \in \Z$, the Baumslag-Solitar group $\BS(m,n) = \langle a,t \; | \; ta^m t^{-1} = a^n \rangle$. These groups are examples of HNN extensions; indeed $\BS(m,n) = \HNN(G,H,\nu)$ where $G = \Z = \langle a \rangle$, $H = m \Z \leq \Z $ and $\nu: m\Z \rightarrow n\Z$ is given by $\nu(mk) = nk$. 

\begin{cor}
\label{Cor: Baumslag Solitar groups SH}
    For non-zero integers $m,n \in \Z$, the Baumslag Solitar groups $\BS(m,n)$ are cocycle-hyperlinear. 
\end{cor}
\begin{proof}
    Looking at $\BS(m,n)$ as $\HNN(\Z,m\Z,\nu)$ as above, we notice that $\Z$ is amenable and cocycle-hyperlinear. Hence Corollary \ref{Prop: HNN extensions SH} applies.  
\end{proof}

\section*{AI statement}

The main ideas for Theorems \ref{Main Thm: amenable-by-treeable} and \ref{Main Thm: central extensions of SL(n,Z)} were obtained by the authors (without any use of LLMs) during their visit to the UK operator algebras conference in Sabhal Mor Ostaig (Isle of Skye, Scotland) in June, 2026. Subsequently ChatGPT 5.5 and later ChatGPT Sol were used in this project and the main contributions were as follows: 
\begin{enumerate}
    \item The authors received substantial help from AI in understanding the main results of van der Kallen's work \cite{VanderKallen}, which led them to concretize their main idea of the proof of Theorem \ref{Main Thm: central extensions of SL(n,Z)}.  
    \item The fact that every group that admits an action with a treeable orbit equivalence relation and amenable stabilizers is in fact an amenable-by-treeable group, was actually proven by AI. The authors proved Theorem \ref{Main Thm: amenable-by-treeable} for groups admitting such actions, and they wanted to understand how much more general this class is from amenable-by-treeable groups. They asked AI to get an example of a group admitting such actions which is not amenable-by-treeable. This ultimately led to AI giving them a rather easy proof of this equivalence (which they suspect might be folklore among some experts). 
    \item While dealing with amalgamated free products, AI gave the authors the reference to the work of Fima and Freslon \cite{Fima_Amaury}, which led them to deduce cocycle-hyperlinearity for HNN extensions. 
\end{enumerate}
ChatGPT Astra was used at the final stage for proof-reading, and the authors take complete responsibility for the mathematical contents of this article.

\section*{Acknowledgments}

S.C. is supported by the ERC advanced grant no. 101141693 titled \textit{Noncommutative ergodic theory of higher rank lattices}. He would like to thank Alon Dogon for a wonderful discussion on cocycle-hyperlinearity in CIRM in 2025, and for some useful comments on this article. F.F. gratefully acknowledges support from the Simons Foundation Dissertation Fellowship SFI-MPS-SDF-00015100. He is also grateful to Professor Ben Hayes for the interesting discussion surrounding the present topic. The authors also thank the organisers of the UK Operator Algebras Conference 2026 in Sabhal Mor Ostaig for hosting them, where many of the preliminary discussions took place.

\printbibliography

\end{document}